\documentclass{amsart}
\usepackage{amsmath}
\usepackage{amsfonts}
\usepackage{amssymb}
\usepackage{amsthm}
\usepackage{mathtools}
\usepackage{url}
\usepackage{tikz-cd}
\usepackage{arydshln}
\usepackage{enumitem}
\usepackage{dsfont}

\newcommand{\NN}{\mathbb{N}}
\newcommand{\ZZ}{\mathbb{Z}}

\newcommand{\U}{\mathcal{U}}
\newcommand{\A}{\mathcal{A}}

\newcommand{\HN}{\prescript{*}{}{\NN}}

\theoremstyle{theorem}
\newtheorem{definition}{Definition}[section]
\newtheorem{theorem}{Theorem}[section]
\newtheorem{corollary}{Corollary}[section]
\newtheorem{conjecture}{Conjecture}[section]
\newtheorem{lemma}{Lemma}[section]

\theoremstyle{definition}
\newtheorem{case}{Case}

\begin{document}
\title{The Unreasonable Rigidity of Ulam Sets}

\author{J. Hinman}
\email[Joshua Hinman]{joshua.hinman@yale.edu}
\author{B. Kuca}
\email[Borys Kuca]{borys.kuca@yale.edu}
\author{A. Schlesinger}
\email[Alexander Schlesinger]{alexander.schlesinger@yale.edu}
\author{A. Sheydvasser}
\email[Arseniy Sheydvasser]{arseniy.sheydvasser@yale.edu}

\address{Department of Mathematics, Yale University, 10 Hillhouse Avenue, New Haven, CT 06511}


\maketitle

\begin{abstract}
We give a number of results about families of Ulam sets. Generalizing behavior of Ulam sets $U(1,n)$, we prove using an novel model theoretic approach that there is a rigidity phenomenon for Ulam sets $U(a,b)$ as $b$ increases. Based on this, we suggest a natural conjecture, and investigate its potential applications, including a method of proving certain families of Ulam sequences are regular, for which we also provide partial, unconditional, results. Along this same vein, we give an upper bound bound on the density of Ulam sequences $U(1,n)$. Finally, we give classification results for higher dimensional Ulam sets.
\end{abstract}

\section{Introduction and Main Results:}\label{Summary}

\subsection{Introduction}
In 1964, Ulam introduced his eponymous sequence
	\begin{align*}
    1, 2, 3, 4, 6, 8, 11, 13, 16, 18, 26, 28, 36, 38, 47, 48, 53, 57, 62, 69, 72, 77, 82, 87, 97 \ldots
    \end{align*}
    
\noindent defined recursively so that the first two terms are $1,2$ and each  subsequent term is the smallest integer that can be written as the sum of two distinct prior terms in a unique way \cite{ulam_1964}---we shall denote this sequence as $U(1,2)$, for reasons that will be evident later. Ulam was interested in determining the growth of this sequence---the best known bound is that the Ulam sequence grows no faster than the Fibonacci sequence. However, the Fibonacci grows exponentially, whereas experimental data suggests that the Ulam sequence has positive density about 0.079.
 
Recently, there has been renewed interest in the Ulam sequence and its generalizations due to a paper of Steinerberger \cite{steinerberger_2016} describing a discovered ``signal" in the Ulam sequence. Specifically, Steinerberger observed that for 
	\begin{align*}
    \lambda \approx 2.443442967\ldots,
    \end{align*}
    
\noindent the distribution of the sequence
	\begin{align*}
    U(1,2) \mod \lambda
    \end{align*}
    
\noindent seems to be concentrated in the middle third of the interval, but is not a discrete distribution. This is a new and unexpected phenomenon; it is an old theorem of Weyl \cite{weyl_1916} that for any integer sequence $\{a_i\}_{i = 1}^\infty$,
	\begin{align*}
    a_i \mod \lambda
    \end{align*}

\noindent is equidistributed for almost all $\lambda$. On the other hand, the Ulam sequence and its generalizations are the only known naturally-defined sequences for which the exceptions to Weyl's theorem give fixed, stable distributions that are not discrete.

	\begin{figure}
	\includegraphics[height = 0.25\textheight]{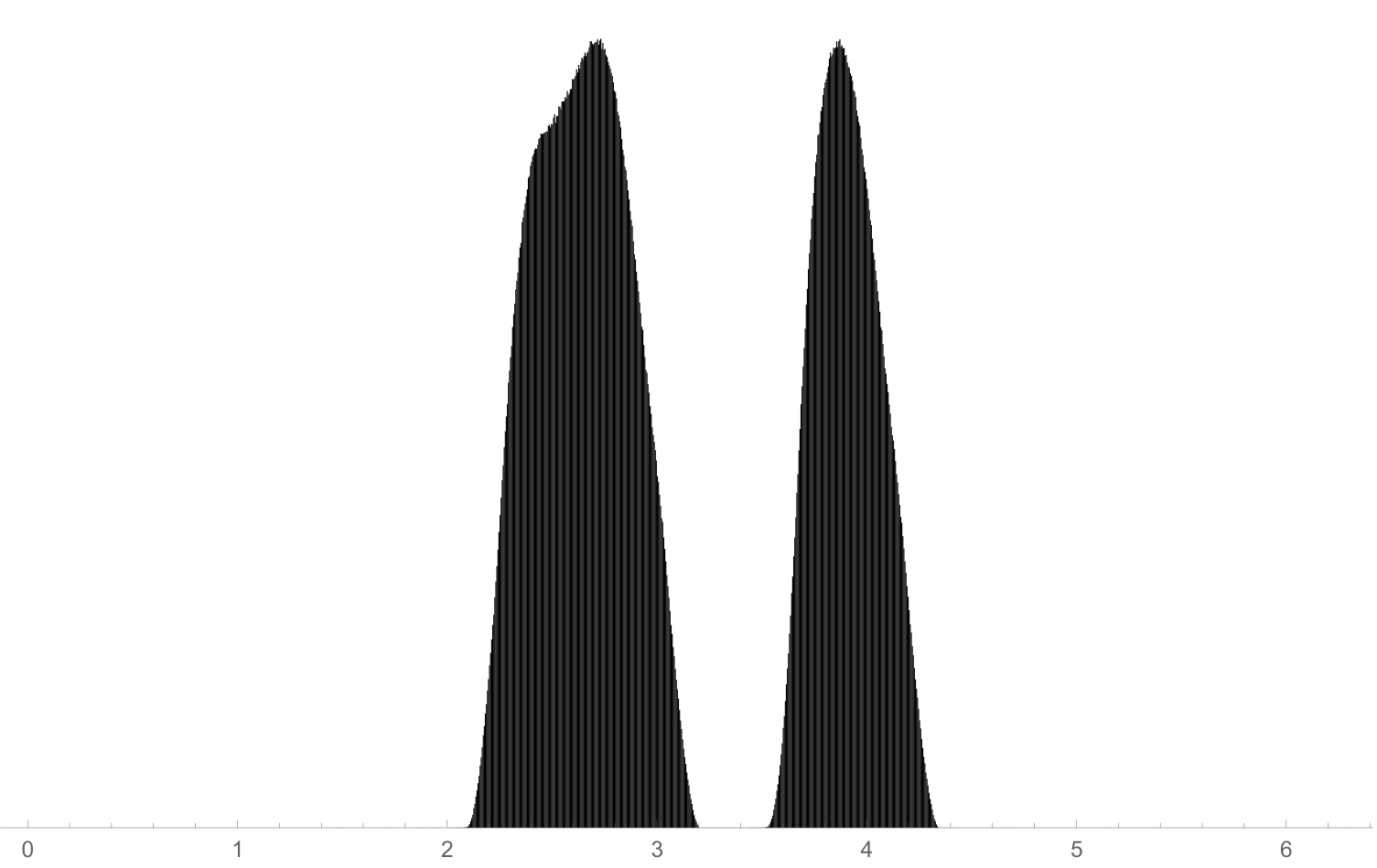}
	
	\caption{The distribution of $U(1,2) \mod \lambda$, plotted with $10^6$ terms.}
	\label{UlamHistogram1_2}
	\end{figure}
    
This same odd behavior can be observed for many similarly defined sequences with different starting conditions. However, the great difficulty in proving anything about Ulam sets is that we understand so little about their structure---they almost seem random, and have been described in the literature as ``erratic". Our present goal is to show that, in contrast, families of Ulam sets varying in some parameter can be startlingly rigid.

\subsection{Summary of Main Results:}
Let $U(a,b)$ denote the generalized Ulam sequences starting with integers $a,b$, such that each subsequent term is the smallest integer that can be written as the sum of two distinct preceding terms in exactly one way. An important class of examples is the family $U(1,n)$ where $n \geq \ZZ_{>1}$. The first few terms of $U(1,2)$, $U(1,3)$, $U(1,4)$, $U(1,5)$, and $U(1,6)$ are given below.
\begin{align*}
\arraycolsep=3pt
    \begin{array}{llllllllllllllllll}
    U(1,2) & = & 1, & 2, & 3, & 4, & 6, & 8, & 11, & 13, & 16, & 18, & 26, & 28, & 36, & 38, & 47, & 48 \ldots \\
    U(1,3) & = & 1, & 3, & 4, & 5, & 6, & 8, & 10, & 12, & 17, & 21, & 23, & 28, & 32, & 34, & 39, & 43 \ldots \\
    U(1,4) & = & 1, & 4, & 5, & 6, & 7, & 8, & 10, & 16, & 18, & 19, & 21, & 31, & 32, & 33, & 42, & 46 \ldots \\
    U(1,5) & = & 1, & 5, & 6, & 7, & 8, & 9, & 10, & 12, & 20, & 22, & 23, & 24, & 26, & 38, & 39, & 40 \ldots \\
    U(1,6) & = & 1, & 6, & 7, & 8, & 9, & 10, & 11, & 12, & 14, & 24, & 26, & 27, & 28, & 29, & 31, & 45 \ldots
    \end{array}
\end{align*}    
    
\noindent Startlingly, there appears to be a simple formula in $n$ for the first few terms of each of these sequences---specifically,
	\begin{align*}
    U(1,n) \cap [1,3n] = \{1\} \cup \{n, n + 1, \ldots, 2n\} \cup \{2n + 2\}.
    \end{align*}
    
\noindent We shall prove this as a lemma in Section \ref{Applications}. However, for $n \geq 4$, this pattern seems to extend further.
	\begin{align*}
    U(1,n) \cap [1,6n] &= \{1\} \cup \{n, n + 1, \ldots, 2n\} \\
    &\cup \{2n + 2\} \cup \{4n\} \\
    &\cup \{4n + 2, 4n + 3, \ldots 5n - 1\} \\
    &\cup \{5n + 1\}.
    \end{align*}
    
\noindent That this is true for all $n \geq 4$ is a consequence of Theorem \ref{Rigidity for a=1} in Section \ref{Special Rigidity Section}. This suggests a conjecture that seems too good to be true, but nevertheless was confirmed by the authors for thousands of terms of $U(1,n)$.

\begin{conjecture}\label{Conjecture That Started It All}
There exist integer coefficients $m_i, p_i, k_i, r_i$ such that for all integers $n \geq 4$,
	\begin{align*}
    U(1,n) = \bigsqcup_{i = 1}^\infty [m_i n + p_i, k_i n + r_i] \cap \ZZ.
    \end{align*}
\end{conjecture}

While at present it is unknown how to prove a result as strong as Conjecture \ref{Conjecture That Started It All}, we construct an extension of the Ulam sequence over the hyperreals to prove a result that is in a sense the next best thing.

\begin{theorem}
There exist integer coefficients $m_i, p_i, k_i, r_i$ such that for any $C > 0$, there exists an integer $N_0$ such that for all integers $N \geq N_0$,
	\begin{align*}
    U(1,N) \cap [1,CN] = \left(\bigsqcup_{i = 1}^\infty [m_i n + p_i, k_i n + r_i] \cap \ZZ\right) \cap [1,CN].
    \end{align*}
\end{theorem}

\noindent We give a more general version of this result for all Ulam sequences $U(a,b)$ in Section \ref{Rigidity Section}. We give further results of this type for the special case $U(1,n)$ in Section \ref{Special Rigidity Section}. This general methodology of studying families of Ulam sequences is logically continued in Section \ref{Applications}, where we give both unconditional results and improvements based on the conjectured rigidity of Ulam sequences.

\begin{theorem}
For integer pairs $(a,b)$ given below, the difference between consecutive terms of $U(a,b)$ are eventually periodic.
	\begin{align*}
    \begin{array}{lllllll}
    (4, 11) & (4, 19) & (6, 7) & (6, 11) & (7, 8) & (7, 10) & (7, 12) \\
    (7, 16) & (7, 18) & (7, 20) & (8, 9) & (8, 11) & (9, 10) & (9, 14) \\
    (9, 16) & (9, 20) & (10, 11) & (10, 13) & (10, 17) & (11, 12) & (11, 14) \\
    (11, 16) & (11, 18) & (11, 20) & (12, 13) & (12, 17) & (13, 14)
    \end{array}
    \end{align*}
\end{theorem}

\begin{theorem}
The density of $U(1,n)$ is bounded above by $\frac{n + 1}{3n}$.
\end{theorem}

We also consider ``Ulam-like" behavior and rigidity in higher dimensions. Using the terminology of Kravitz and Steinerberger \cite{kravitz_steinerberger_2017}, we define Ulam sets as follows.

	\begin{definition}\label{UlamSetDefinition}
    Let $|\cdot|$ be a norm on $\ZZ^n$ that increases monotonically in each coordinate. A $(k,n)$-\emph{Ulam set} $U\left(v_1,v_2,\ldots v_k\right)$ is a recursively defined set that contains $v_1, v_2,\ldots v_n \in \ZZ_{\geq 0}^n$ and each subsequent vector is the vector of smallest norm that can be written as a sum of two distinct vectors in the set in exactly one way. We shall say $U\left(v_1,v_2,\ldots v_k\right)$ is \emph{non-degenerate} if $v_i \notin U\left(v_1, v_2, \ldots v_{i - 1}, v_{i + 1}, \ldots v_k\right)$ for every $1 \leq i \leq k$.
    \end{definition}
    
\noindent Two remarks are necessary here: first, it may appear that the definition of Ulam set depends on the choice of monotonically increasing norm $|\cdot|$. In fact, this is not so, as proved in \cite{kravitz_steinerberger_2017}. Secondly, it may be unclear which vector is added if there is more than one of equal norm. However, by the above, this is irrelevant.

Contingent on some natural restrictions described in Section \ref{HigherDimensions}, we classify all $(3,2)$-Ulam sets, showing that they necessarily belong to one of a finite number of different types, illustrated in Figure \ref{UlamTypes}.

\begin{figure}
\begin{tabular}{c}
\begin{tabular}{ccc}
\includegraphics[height = 0.1\textheight]{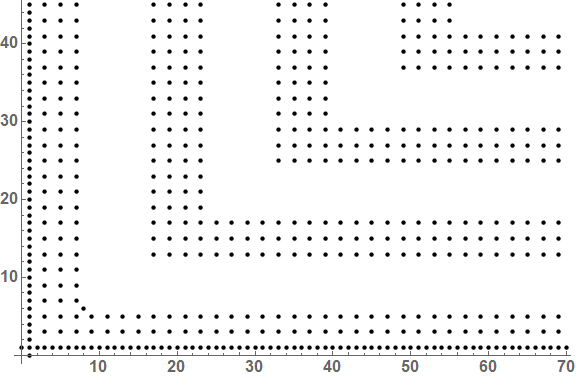} & \includegraphics[height = 0.1\textheight]{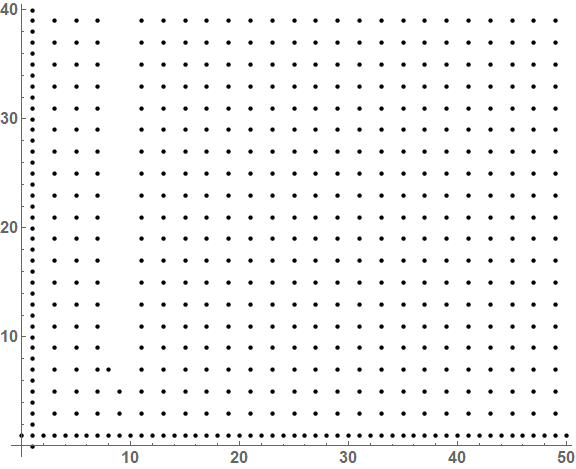} & \includegraphics[height = 0.1\textheight]{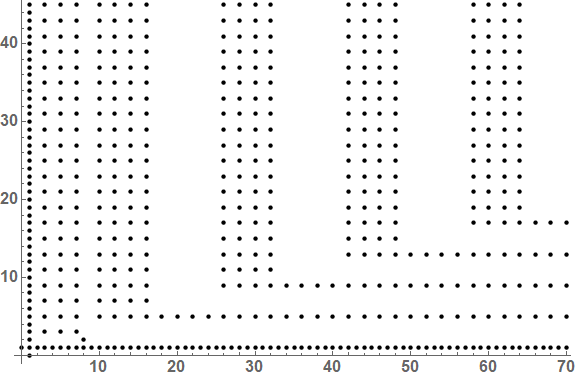}  \\
$U_\A(8,6)$ & $U_\A(8,7)$ & $U_\A(8,2)$
\end{tabular} \\ \\
\begin{tabular}{cc}
\includegraphics[height = 0.1\textheight]{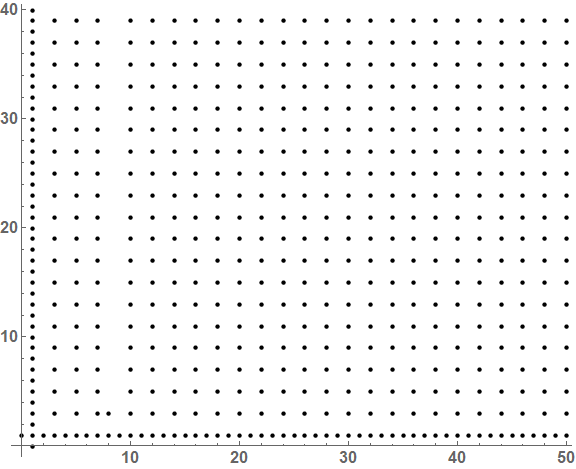} & \includegraphics[height = 0.1\textheight]{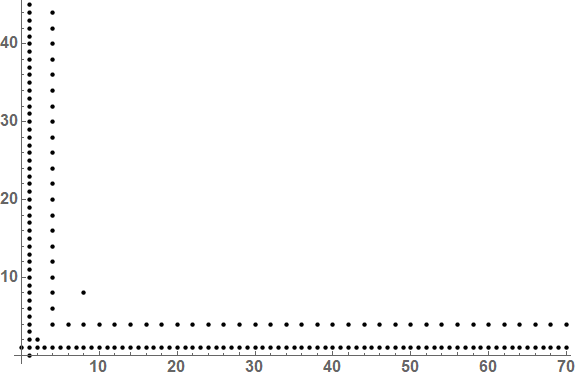} \\
$U_\A(8,3)$ & $U_\A(2,2)$
\end{tabular}
\end{tabular}
\caption{From left to right and top to bottom: sets $U_\A(v_1, v_2)$ of $L$, column-deleted, column-deleted $L$, shifted column-deleted, and exceptional type.}
\label{UlamTypes}
\end{figure}

\begin{theorem}
Let $\U = U\left((1,0), (0,1), (v_1, v_2)\right)$ be a non-degenerate $(3,2)$-Ulam set such that $v_1, v_2 \neq 0$. Then exactly one of the following is true of either $\U$ or its reflection about the $y = x$ line.
\begin{enumerate}
	\item $v_1, v_2 \in 2\ZZ \cap [4,\infty)$ and $\U$ is of $L$ type.
    \item $v_1 \in 2\ZZ$, $v_2 \in \left(1 + 2\ZZ\right) \cap [4,\infty)$, and $\U$ is of column-deleted type.
    \item $v_1 \in 2\ZZ \cap [4,\infty)$, $v_2 = 2$, and $\U$ is of column-deleted $L$ type.
    \item $v_1 \in 2\ZZ$, $v_2 = 3$, and $\U$ is of shifted column-deleted type.
    \item $v_1 = v_2 = 2$ and $\U$ is of exceptional type.
\end{enumerate}
\end{theorem}

See Section \ref{HigherDimensions} for definitions of the various types of Ulam sets. Finally, in Section \ref{Parity Section} we show that there is a parity restriction on more general $(k,2)$-Ulam sets.

\begin{theorem}
Let $\U = U\left((1,0),(0,1),v_1, v_2, \ldots v_n\right)$ be a non-degenerate $(n + 2, 2)$-Ulam set such that none of the $v_i$ lie on the coordinate axes. Then there exists a $(w_1, w_2) \in \ZZ_{\geq 0}^2$ such that for all $(m,n) \in \U$, if $m \geq w_1, n \geq w_2$, then $m = w_1 \mod 2$, $n = w_2 \mod 2$.
\end{theorem}

\subsection*{Acknowledgements} The authors would like to thank Stefan Steinerberger, both for suggesting this problem and for giving valuable feedback along the way. They are indebted to Yuxuan Ke for coding help. Finally, they thank the organizers of SUMRY 2017, where this paper took shape.

\section{Rigidity in $(2,1)$-Ulam Sets:}\label{Rigidity Section}
Our goal in this section is to show that if we fix $a \in \ZZ_{> 0}$ and vary $b$ in some congruence class, then the sets $U(a,b)$ are rigid in a strong sense: there exists a ``nice" function
	\begin{align*}
    \text{Ulam}: \NN \rightarrow P(\NN)
    \end{align*}
    
\noindent such that for all $C > 0$, there exist positive integers $N_0, L$ such that for all $N \geq N_0$ with $N \equiv N_0 \mod L$,
	\begin{align*}
    \text{Ulam}(N) \cap [1,CN] = U(a,N) \cap [1,CN].
    \end{align*}
    
\noindent Before we give a more precise statement and a proof, we first give a slightly different characterization of $(1,2)$-Ulam sets than the one we have been using up until now---in particular, we are going to show that $(1,2)$-Ulam sets are first order axiomatizable.

\begin{lemma}\label{FirstOrderDescription}
Let $P(n)$ be a predicate on $\NN$ satisfying the following properties.
	
	\begin{enumerate}
		\item $\forall n$, if $n \leq b$ and $P_{a,b}(n)$, then $n = a$ or $n = b$.
		\item $\forall n$, if $P_{a,b}(n)$ and $n > b$ then $\exists ! u<v$ such that $P_{a,b}(u)$, $P_{a,b}(v)$, and $n = u + v$.
		\item $\forall n$, if $P_{a,b}(n)$ and $\exists N > n$ such that $\forall r \in (n,N)$, $\neg P_{a,b}(r)$ and there exists a unique pair $u<v$ such that $N = u + v$ with $P_{a,b}(u) and P_{a,b}(v)$, then $P_{a,b}(N)$.
	\end{enumerate}
		
\noindent Then $P_{a,b}(n)$ if and only if $n \in U(a,b)$.
\end{lemma}
	
\begin{proof}
We prove this by induction. The base case where $n \leq b$ is evident, so we assume that $\forall n < k$, $P_{a,b}(k)$ if and only if $k \in U(a,b)$, and we try to prove it for $n$.
	
Let $k'$ be the largest element of $U(a,b)$ smaller than $n$. If $P_{a,b}(n)$, then $n = u + v$ for some unique $u<v$ such that $P_{a,b}(u), P_{a,b}(v)$---by the inductive hypothesis, we know that this is the same as saying that there are unique $u<v \in U(a,b)$ such that $n = u + v$. Therefore, $n$ is the smallest integer greater than $k'$ that has a unique representation as the sum of two distinct Ulam numbers; ergo, $n \in U(a,b)$.
	
Conversely, if $n \in U(a,b)$, then $n = u + v$ for some unique $u < v$ such that $P_{a,b}(u)$ and $P_{a,b}(v)$. Since $P_{a,b}(k')$ and $\forall r \in (k,n)$, $\neg P_{a,b}(r)$, we conclude that $P_{a,b}(n)$.
\end{proof}
	
In light of Lemma \ref{FirstOrderDescription}, it makes perfect sense to generalize $(1, 2)$-Ulam sets from subsets of the naturals to subsets of ordered abelian groups.

\begin{definition}
Let $A$ be an ordered abelian group, and $0 < a < b \in A$. A subset $U$ of $A$ is an \emph{Ulam subset
with respect to} $a$, $b$ if there is some predicate $P_{a,b}$ on A satisfying the first order axioms given in Lemma \ref{FirstOrderDescription} such that $u \in U$ if and only if $P_{a,b}(u)$ is true.
\end{definition}
	
The authors believe that this description might be of independent interest in the study of how Ulam sets can be generalized to more general abelian groups. However, we are primarily interested in the case that $A = \HN$, the hypernaturals, where Ulam subsets are quite structured---a fact that we shall exploit. We start with the following example.

\begin{lemma}\label{Initial Segment}
Let $a \in \NN$, and let $\omega$ be any hypernatural larger than every standard integer that is coprime to $a$. If $U$ is the Ulam subset of $\HN$ with respect to $a,\omega$, then
	\begin{align*}
	U \cap [a,2\omega + a] = \{a\} \cup \left([\omega,2\omega + a - 1] \cap \left(\omega + a\HN\right)\right) \cup \{2\omega + a\}.
	\end{align*}
\end{lemma}
	
\begin{proof}
It is obvious that up to $\omega$, $a,\omega$ are the only elements of $U$. It is similarly clear that $\omega + 1, \omega + 2, \ldots \omega + a - 1$ cannot be written as sums of prior terms. On the other hand, $\omega + a$ has a unique decomposition, hence $\omega + a \in U$. Note that for all $\omega < \omega' < 2\omega + a$, any decomposition in terms of prior terms must be of the form $a + (\omega' - a)$---therefore, $\omega' \in U$ if and only if $\omega' - a \in U$. Noting that $U \cap [1,2\omega + a - 1]$ is a hyperfinite set, and $2\omega + a = \omega + (\omega + a)$, this proves the claim.
\end{proof}
	
Using this machinery, we can now give a precise description of the function $\text{Ulam}$. Let $m_1, m_2$ be integers, let $s^{(l)} = \{s_i\}_{i = 0}^{l - 1}$ be a binary sequence of length $l$, and define the set
	\begin{align*}
	A\left(m_1, m_2, s^{(l)}\right) = \left\{n \in \ZZ_{\geq 0} \middle| n \in [m_1,m_2], \ s^{(l)}_{n - m_1 \mod l} = 1\right\}.
	\end{align*}

\begin{theorem}\label{RigidityTheorem}
Let $a$ be a positive integer. For every positive integer $C$, there exist integers $L,N_0 \geq 1$ such that for every congruence class $c \mod L$, if $N \geq N_0$ and $N \equiv c \mod L$, we can decompose $U(a,N)$ as a disjoint union,
	\begin{align*}
	U(a,N) \cap [1,CN] = \left(\bigcup_{1}^\infty\bigsqcup_{i_l = 1}^\infty A\left(m_{i_l}N + p_{i_l}, k_{i_l}N + r_{i_l},s^{(l)}_{i_l}\right)\right) \cap [1,CN],
	\end{align*}
			
\noindent such that $m_{i_l} = k_{i_l}$ if and only if $l = 1$. Furthermore, the integer coefficients $m_{i_l}, k_{i_l}, r_{i_l}, s_{i_l}$ and the binary sequences $s^{(l)}_{i_l}$ depend only on the congruence class $c$.
\end{theorem}
	
\begin{proof}
The general approach is to prove this statement with $N_0,N$ replaced with hypernaturals $\omega,\omega'$ larger than any standard natural, and then use the transfer principle to conclude that the original statement over the naturals is also true.
	
We fix a hypernatural $\omega$ larger than any standard natural. We shall show that we can decompose
	\begin{align*}
	U(a,\omega) \cap [1,C\omega] = \left(\bigcup_{l = 0}^\infty\bigsqcup_{i_l = 1}^\infty A\left(m_{i_l}\omega + p_{i_l}, k_{i_l}\omega + r_{i_l},s^{(l)}_{i_l}\right)\right) \cap [1,C\omega],
	\end{align*}
			
\noindent such that the coefficients are all integers and $m_{i_l} = p_{i_l}$ if and only if $l = 1$. Our approach is to consider a recursive algorithm for computing the sets
	\begin{align*}
	A\left(m_{i_l}\omega + p_{i_l}, k_{i_l}\omega + r_{i_l},s^{(l)}_{i_l}\right),
	\end{align*}
		
\noindent defined as follows: given that we have decomposed $U(a,\omega) \cap [1, c\omega + d]$ for some integers $c,d$ into sets of the desired form, choose the smallest element $u_1 \in U(a,\omega)$ larger than $c\omega + d$---we shall show that $u_1 = c'\omega + d'$ for some integers $c',d'$. As an aside, note that there always exists a smallest element larger than a given element of $U(a,\omega)$, since $U(a,\omega)$ is hyperfinite. If for all $n \in \NN$ the binary sequence
	\begin{align*}
	s_n = \begin{cases} 1 & \text{if } c'\omega + d' + n \in U(a,\omega) \\ 0 & \text{if } c'\omega + d' + n \notin U(a,\omega) \end{cases}
	\end{align*}
		
\noindent is periodic for some minimal period $l$, then we choose the smallest element $u_2$ of $U(a,\omega)$ such that
	\begin{align*}
	A\left(c'\omega + d', u_2, s_n^{(l)}\right) \subset U(a,\omega) \cap [1,C\omega],
	\end{align*}
		
\noindent and adjoin it to the existing decomposition---we shall similarly show that $u_2 = c''\omega + d''$ for some integers $c'', d''$. Otherwise, we add the set
	\begin{align*}
	A\left(c'\omega + d', c'\omega + d', 1\right).
	\end{align*}
		
\noindent To prove that this algorithm is well-defined and halts, we need to show three things. First, we must show that the smallest element $u_1$ in $U(a,\omega)$ larger than $c\omega + d$ is of the form $c'\omega + d'$. Secondly, if the binary sequence $s_n$ becomes periodic, then the largest element $u_2$ in $U(a,\omega)$ such that
	\begin{align*}
	A\left(c'\omega + d', u_2, s_n^{(l)}\right) \subset U(a,\omega) \cap [1,C\omega]
	\end{align*}
		
\noindent is of the form $u_2 = c''\omega + d''$. Finally, for any integer $c > 0$, there exists an integer $d$ such that $s_n$ becomes periodic at $c\omega + d$.
	
If all of these criteria are met, then it is clear that the given algorithm will produce a decomposition of $U(a,\omega)$ of the desired type in finite time---it must eventually halt, since there are only finitely many integers less than $C$, and for every $c \leq C$ there are only finitely many sets $A$ that start at $c\omega + d$ for some $d$. We already know by Lemma \ref{Initial Segment} that we can start with a decomposition up to $2\omega + a$.
	
Suppose we have a decomposition of $U(a,\omega)$ up to $c\omega + d$, and let $u_1$ be the smallest element of $U(a,\omega)$ larger than $c\omega + d$. If $u_1 = c\omega + d'$ for some integer $d'$, then the first criterion is met automatically. Otherwise, we note that $u_1 = a_1 + a_2$ where
	\begin{align*}
	a_1 \in A_1 &= A\left(m_{i_{l_1}}\omega + p_{i_{l_1}}, k_{i_{l_1}}\omega + r_{i_{l_1}},s^{(l_1)}_{i_{l_1}}\right) \\
	a_2 \in A_2 &= A\left(m_{i_{l_2}}\omega + p_{i_{l_2}}, k_{i_{l_2}}\omega + r_{i_{l_2}},s^{(l_1)}_{i_{l_2}}\right) .
	\end{align*}
		
\noindent If there exists an integer $N > 0$ such that
	\begin{align*}
	a_1 \in A_1 \cap \left([m_{i_{l_1}}\omega + p_{i_{l_1}}, m_{i_{l_1}}\omega + p_{i_{l_1}} + N] \cup [k_{i_{l_1}}\omega + r_{i_{l_1}} - N, k_{i_{l_1}}\omega + r_{i_{l_1}}]\right) \\
	a_2 \in A_2 \cap \left([m_{i_{l_2}}\omega + p_{i_{l_2}}, m_{i_{l_2}}\omega + p_{i_{l_2}} + N] \cup [k_{i_{l_2}}\omega + r_{i_{l_2}} - N, k_{i_{l_2}}\omega + r_{i_{l_2}}]\right),
	\end{align*}
		
\noindent then it is clear that $u_1 = a_1 + a_2$ can be written as $c'\omega + d'$ for some integers $c',d'$. Otherwise, without loss of generality, $a_2 - N \in U(a,\omega)$ for some integer $N > 0$ which we can take to be a multiple of every period $l$ found thus far in the decomposition. We have that $u_1 - N = a_1 + (a_2 - N)$ is a partition into elements of the Ulam sequence, but we know that $u_1 - N \notin U(a,\omega)$ for any integer $N > 0$, and so there must exist a second partition $u_1 - N = a_1' + a_2'$ such that $a_1', a_2'$ are distinct elements of $U(a,\omega)$.
	
This implies $u_1 = a_1' + a_2' + N$---for this not to give a second partition of $u_1$, it must be that $a_1' + N_1, a_2' + N_2 \notin U(1,\omega)$ for any $N_1,N_2 > 0$ such that $N = N_1 + N_2$. However, we know that
	\begin{align*}
	a_1' \in A_1' &= A\left(m_{i_{l_1'}}\omega + p_{i_{l_1'}}, k_{i_{l_1'}}\omega + r_{i_{l_1'}},s^{(l_1')}_{i_{l_1'}}\right) \\
	a_2' \in A_2' &= A\left(m_{i_{l_2'}}\omega + p_{i_{l_2'}}, k_{i_{l_2'}}\omega + r_{i_{l_2'}},s^{(l_1')}_{i_{l_2'}}\right),
	\end{align*}
		
\noindent and since we chose $N$ so that it is a multiple of every period $l$ found thus far in the decomposition, the only way that it is possible that $a_1' + N_1, a_2' + N_2 \notin U(1,\omega)$ for all possible choices of $N_1$ and $N_2$ is if there exists an integer $M > 0$ such that
	\begin{align*}
	a_1' \in A_1' \cap \left([m_{i_{l_1'}}\omega + p_{i_{l_1'}}, m_{i_{l_1'}}\omega + p_{i_{l_1'}} + M] \cup [k_{i_{l_1'}}\omega + r_{i_{l_1'}} - M, k_{i_{l_1'}}\omega + r_{i_{l_1'}}]\right) \\
	a_2' \in A_2' \cap \left([m_{i_{l_2'}}\omega + p_{i_{l_2'}}, m_{i_{l_2}'}\omega + p_{i_{l_2}'} + M] \cup [k_{i_{l_2}'}\omega + r_{i_{l_2}'} - M, k_{i_{l_2}'}\omega + r_{i_{l_2}'}]\right),
	\end{align*}
		
\noindent and since $u_1 = a_1' + a_2' + N$, we conclude that it can be written in the form $c'\omega + d'$ for some integers $c',d'$.
	
Next, we prove that for every $c\omega + d \in U(a,\omega)$, there exists an integer $d' \geq d$ such that the binary sequence $s_n$ is periodic. Equivalently, we must show that the function
	\begin{align*}
	R: \left\{c\omega + d' \middle| d' \in \NN, d' \geq d\right\} &\rightarrow \{0,1,\infty\}
	\end{align*}
		
\noindent that counts the number of partitions of $c\omega + d'$ as distinct elements of the Ulam sequence is eventually periodic, round up to $\infty$ if there is more than one partition. We define a binary operation $+$ on $\{0,1,\infty\}$ in the obvious way, rounding up to $\infty$ if the sum is greater than $1$.
	
By construction, there are only finitely many pairs of sets
	\begin{align*}
	A_1 &= A\left(m_{i_{l_1}}\omega + p_{i_{l_1}}, k_{i_{l_1}}\omega + r_{i_{l_1}},s^{(l_1)}_{i_{l_1}}\right) \\
	A_2 &= A\left(m_{i_{l_2}}\omega + p_{i_{l_2}}, k_{i_{l_2}}\omega + r_{i_{l_2}},s^{(l_1)}_{i_{l_2}}\right)
	\end{align*}
		
\noindent already in the set such that there exist elements $a_1 \in A_1, a_2 \in A_2$ such that $a_1 + a_2 \in c\omega + d + \NN$---let the set of such pairs be denoted by $\mathcal{A}$. Note that for any pair $(A_1, A_2) \in \mathcal{A}$, the function
	\begin{align*}
	f_{A_1,A_2}: \left\{c\omega + d' \middle| d' \in \NN, d' \geq d\right\} &\rightarrow \{0,1,\infty\} \\
	c\omega + d' &\mapsto \sum_{\substack{a_1 + a_2 = c\omega + d' \\ (a_1,a_2) \in (A_1,A_2)}} 1 \\
	&= \sum_{a_1 \in A_1} 1_{c\omega + d' - a_1 \in A_2}
	\end{align*}
		
\noindent is eventually periodic, as it is the sum of functions that are eventually periodic. From this, it follows that the function
	\begin{align*}
	F: \left\{c\omega + d' \middle| d \in \NN, d \geq d'\right\} &\rightarrow \{0,1,\infty\} \\
	c\omega + d' &\mapsto \sum_{(A_1, A_2) \in \mathcal{A}} f_{A_1, A_2}(c\omega + d')
	\end{align*}
		
\noindent must eventually be periodic. Note that
	\begin{align*}
	R(c\omega + d') = \begin{cases} F(c\omega + d') & \text{if } d' - d \leq a \\ F(c\omega + d') + 1_{c\omega + d' - a \in U(a,\omega)} & \text{otherwise} \end{cases}.
	\end{align*}
		
\noindent For sufficiently large $d'$, $c\omega + d' - a \in U(a,\omega)$ if and only if $F(c\omega + d' - a) = 1$ and $c\omega + d' - 2a \notin U(a,\omega)$ or $F(c\omega + d' - a) = 0$ and $c\omega + d' - 2a \in U(a,\omega)$. Consequently, $R(c\omega + d')$ is eventually periodic, as desired.
	
Finally, it remains to show that the largest element $u_2$ such that
	\begin{align*}
	A\left(c\omega + d, u_2, s_n^{(l)}\right) \subset U(a,\omega) \cap [1,C\omega]
	\end{align*}
			
\noindent can be written in the form $u_2 = c'\omega + d'$ for some integers $c',d'$. If $u_2 = C\omega$, we are done. Otherwise, we note that $u_2$ is of the desired form if and only if $u_2 - l'$ for all $0 \leq l' \leq l$ is of the desired form. So, we choose $u_2 - l'$ such that exactly one of $u_2 - l', u_2 + l - l'$ is in $U(a,\omega)$.
	
We begin by assuming $u_2 - l' \in U(a,\omega)$. Consequently, $u_2 - l' = a_1 + a_2$, where
	\begin{align*}
	a_1 \in A_1 \cap \left([m_{i_{l_1}}\omega + p_{i_{l_1}}, m_{i_{l_1}}\omega + p_{i_{l_1}} + N] \cup [k_{i_{l_1}}\omega + r_{i_{l_1}} - N, k_{i_{l_1}}\omega + r_{i_{l_1}}]\right) \\
	a_2 \in A_2 \cap \left([m_{i_{l_2}}\omega + p_{i_{l_2}}, m_{i_{l_2}}\omega + p_{i_{l_2}} + N] \cup [k_{i_{l_2}}\omega + r_{i_{l_2}} - N, k_{i_{l_2}}\omega + r_{i_{l_2}}]\right).
	\end{align*}
	
\noindent If there exists an integer $N > 0$ such that
	\begin{align*}
	a_1 \in A_1 \cap \left([m_{i_{l_1}}\omega + p_{i_{l_1}}, m_{i_{l_1}}\omega + p_{i_{l_1}} + N] \cup [k_{i_{l_1}}\omega + r_{i_{l_1}} - N, k_{i_{l_1}}\omega + r_{i_{l_1}}]\right) \\
	a_2 \in A_2 \cap \left([m_{i_{l_2}}\omega + p_{i_{l_2}}, m_{i_{l_2}}\omega + p_{i_{l_2}} + N] \cup [k_{i_{l_2}}\omega + r_{i_{l_2}} - N, k_{i_{l_2}}\omega + r_{i_{l_2}}]\right),
	\end{align*}
		
\noindent then we are done. Otherwise, without loss of generality, $a_2 + l \in U(a,\omega)$, and therefore $u_2 - l' + l = a_1 + (a_2 + l)$ is a partition into elements of the Ulam sequence. Since $u_2 - l' + l \notin U(a,\omega)$, there must exist a second partition $u_2 - l' + l = a_1' + a_2'$ such that $a_1', a_2'$ are distinct elements of $U(a,\omega)$.
	
This implies $u_2 - l' = a_1' + a_2' - l$---for this not to give a second partition of $u_2 - l'$, it must be that $a_1' - N_1, a_2' - N_2 \notin U(1,\omega)$ for any $N_1,N_2 > 0$ such that $l = N_1 + N_2$. However, we know that
	\begin{align*}
	a_1' \in A_1' &= A\left(m_{i_{l_1'}}\omega + p_{i_{l_1'}}, k_{i_{l_1'}}\omega + r_{i_{l_1'}},s^{(l_1')}_{i_{l_1'}}\right) \\
	a_2' \in A_2' &= A\left(m_{i_{l_2'}}\omega + p_{i_{l_2'}}, k_{i_{l_2'}}\omega + r_{i_{l_2'}},s^{(l_1')}_{i_{l_2'}}\right),
	\end{align*}
		
\noindent and since by construction $l$ is a multiple of $l_1', l_2'$, the only way that it is possible that $a_1' - N_1, a_2' - N_2 \notin U(1,\omega)$ for all possible choices of $N_1$ and $N_2$ is if there exists an integer $M > 0$ such that
	\begin{align*}
	a_1' \in A_1' \cap \left([m_{i_{l_1'}}\omega + p_{i_{l_1'}}, m_{i_{l_1'}}\omega + p_{i_{l_1'}} + M] \cup [k_{i_{l_1'}}\omega + r_{i_{l_1'}} - M, k_{i_{l_1'}}\omega + r_{i_{l_1'}}]\right) \\
	a_2' \in A_2' \cap \left([m_{i_{l_2'}}\omega + p_{i_{l_2'}}, m_{i_{l_2}'}\omega + p_{i_{l_2}'} + M] \cup [k_{i_{l_2}'}\omega + r_{i_{l_2}'} - M, k_{i_{l_2}'}\omega + r_{i_{l_2}'}]\right),
	\end{align*}
		
\noindent which settles the matter.
	
The second case that $u_2 - l' + l \in U(a,\omega)$, $u_2 - l' \notin U(a,\omega)$ proceeds in the same way. Ergo,
	\begin{align*}
	U(a,\omega) \cap [1,C\omega] = \left(\bigcup_{l = 0}^\infty\bigsqcup_{i_l = 1}^\infty A\left(m_{i_l}\omega + p_{i_l}, k_{i_l}\omega + r_{i_l},s^{(l)}_{i_l}\right)\right) \cap [1,C\omega],
	\end{align*}
		
\noindent as desired. To conclude the proof, we note that only finitely many sets
	\begin{align*}
	A\left(m_{i_l}\omega + p_{i_l}, k_{i_l}\omega + r_{i_l},s^{(l)}_{i_l}\right)
	\end{align*}
		
\noindent can intersect $[1,C\omega]$, and so we can take $L$ to be the lowest common multiple of all the periods $l$ corresponding to such sets. The coefficients $m_i, p_i, k_i, r_i$ and the binary sequences $s_i^{(l)}$ can only depend on the congruence class of $\omega$ modulo $L$, since this is the only information used by the decomposition algorithm to compute these coefficients.
	
We have shown that there exists $\omega \in \HN$ with the property that for all $\omega' \geq \omega$ such that $\omega' \equiv \omega \mod L$,
	\begin{align*}
	U(a,\omega') \cap [1,C\omega'] = \left(\bigcup_{l = 0}^\infty\bigsqcup_{i_l = 1}^\infty A\left(m_{i_l}\omega' + p_{i_l}, k_{i_l}\omega' + r_{i_l},s^{(l)}_{i_l}\right)\right) \cap [1,C\omega'].
	\end{align*}
		
\noindent This statement can be captured by first order logic, and therefore we know that by the transfer principle it must also be true over the naturals. This concludes the proof.
\end{proof}

A priori, there is nothing preventing $N_0, L \rightarrow \infty$ as $C \rightarrow \infty$---Theorem \ref{RigidityTheorem} therefore can only be used to prove statements that involve some initial segment of $(2,1)$-Ulam sets. While this is remarkable in and of itself, the numerical data suggests that something far, far stronger is true.

\begin{conjecture}[Rigidity Conjecture]\label{Rigidity Conjecture}
Let $a$ be a positive integer. There exists an integer $L \geq 1$ such that for every $c \in \left(\ZZ/L\ZZ\right)^\times$ there exists a positive integer $N_0 \equiv c \mod L$ with the property that $\forall N \geq N_0$ satisfying $N \equiv c \mod L$, we can decompose $U(a,N)$ as a disjoint union,
	\begin{align*}
	U(a,N) = \left(\bigcup_{l = 1}^\infty\bigsqcup_{i_l = 1}^\infty A\left(m_{i_l}N + p_{i_l}, k_{i_l}N + r_{i_l},s^{(l)}_{i_l}\right)\right),
	\end{align*}
			
\noindent such that $m_{i_l} = k_{i_l}$ if and only if $l = 1$.
\end{conjecture}
    
\section{Rigidity Results for $U(1,n)$}\label{Special Rigidity Section}
We now examine the special case $U(1,n)$ more closely. We will show that in this case we can prove rigidity results like Theorem \ref{RigidityTheorem}, but with two key differences: we will show that we can explicitly take $L = 1$, and unlike the methodology of the previous section, the proofs are constructive, with bounds on the coefficients.

First, note that there is a unique decomposition
\begin{align*}
U(1,n) = \bigsqcup_{i = 1}^\infty A_i(n)
\end{align*}

\noindent where for each $i$,
\begin{align*}
A_i(n) &= A\left(a_i(n), b_i(n), 1^{(1)}\right) \\
	&= \left[a_i(n), b_i(n)\right] \cap \ZZ,
\end{align*}

\noindent for some integers $b_{i - 1} < a_i \leq b_i < a_{i + 1}$. We call the $A_i$ the \emph{intervals} of $U(1,n)$. If $a_i = b_i$, we shall call $A_i$ an \emph{isolated point}. Otherwise, we call $A_i$ a \emph{long interval} with endpoints $a_i, b_i$, and we call $A_i \ \backslash \ \{a_i, b_i\}$ the \emph{interior} of $A_i$. Based on numerical data, we make the following conjecture.

\begin{conjecture}\label{NumericalRigidity}
There exists an integer $N_0 > 1$ such that for all $n \geq N_0$, the coefficients $a_i(n)$ and $b_i(n)$ are linear functions in $n$ with bounded coefficients---to be precise,
\begin{align*}
	a_i(n) &= (n + B)m_i + \varepsilon_i \\
    b_i(n) &= (n + B)p_i + \delta_i,
\end{align*}

\noindent where $m_i, p_i, \varepsilon_i, \delta_i$ do not depend on $n$, $B, \varepsilon, \delta > 0$ are real constants, $m_i, p_i$ are integers, and $|\varepsilon_i| < \epsilon, |\delta_i| < \delta$.
\end{conjecture}

Numerical data suggests that we can take $N_0 = 4$, $B \approx 0.139$, and $\epsilon, \delta 
\leqslant 2.5$. However, much like Conjecture \ref{Rigidity Conjecture}, we cannot prove this result unconditionally. We can prove, on the other hand, that if the intervals of some $U(1,N_0)$ satisfy some growth conditions, then the intervals of $U(1,N)$ with $N \geq N_0$ satisfy those growth conditions.

We shall say that $U(1,n)$ is $(M,B,\varepsilon,\delta)$-\emph{rigid with coefficients} $m_i, p_i$ if
	\begin{align*}
    n > 4\max(\varepsilon,\delta)-B+4
    \end{align*}
    
\noindent and for all $1 \leq i \leq M$,
	\begin{align*}
    a_i(N) &= (N + B)m_i + \varepsilon_i \\
    b_i(N) &= (N + B)p_i + \delta_i,
    \end{align*}
    
\noindent with $|\varepsilon_i| < \varepsilon, |\delta_i| < \delta$, where $N \in \{n, n - 1\}$ and $a_i, b_i$ are the endpoints of the $i$-th interval of $U(1,N)$.

\begin{theorem}\label{Rigidity for a=1}
Suppose $U(1,N_0)$ is $(M,B,\varepsilon,\delta)$-rigid with coefficients $m_i, p_i$. Then for all $N > N_0$, $U(1,N)$ is $(M,B,\varepsilon,\delta)$-rigid with coefficients $m_i, p_i$.
\end{theorem}

To prove this theorem, we shall need several preliminary lemmas. To start, we describe how the coefficients $a_i, b_i$ depend on coefficients $a_j, b_j$ for $1 \leq i < j$.

\begin{lemma}\label{RecursiveIntervals}
Let $A_i = [a_i,b_i] \cap \ZZ$ denote the $i$-th interval of $U(1,n)$. Then the endpoint $a_i$ is of the form $a_j + a_k + c$ or $ b_j + b_k + c$ for some $1 \leq j \leq k < i$ and integer $c$, where $|c| \leq 2$.
\end{lemma}

\begin{proof}
We know $a_i = u + v$ for some $u,v \in U(1,n)$. If $u = a_j$ for some $j$ which is not an isolated point, then either $v = a_k$ for some $k$, or $v = a_j + 1, a_j + 2$. Otherwise, $a_i = (u + 1) + (v - 1)$ gives a second representation of $a_i$. If $u = b_j$ for some $j$ which is not an isolated point, then either $v = b_k$ for some $k$, or $v = b_k - 1, b_k - 2$. Otherwise, $a_i = (u - 1) + (v + 1)$ gives a second representation of $a_i$. If $u$ is in the interior of an interval, then either $u$ and $v$ are within distance $1$ of the endpoints of that interval, or $v$ is an isolated point. Otherwise, one of $a_i = (u - 1) + (v + 1) = (u + 1) + (v - 1)$ gives a second representation of $a_i$.

By symmetry, the only case we have left to consider is where $u$ is an isolated point, and $v$ is in the interior of an interval. Note first that $a_i - 1 = u + (v - 1)$, hence it must have a second representation $a_i - 1 = u' + v'$, since $a_i \notin U(1,n)$. This implies $a_i = u' + v' + 1$, so for this not to give a second representation of $a_i$, it must be true that $u' + 1, v' + 1 \notin U(1,n)$. Consequently, $u' = b_{j'}$, $v' = b_{k'}$, and we have the desired decomposition $a_i = b_{j'} + b_{k'} + 1$.
\end{proof}

Next, we show that if the conditions of Theorem \ref{Rigidity for a=1} hold for consecutive Ulam sequences, then the bounds on $\epsilon_i, \delta_i$ imply bounds on the endpoints.

\begin{lemma}\label{EndpointBounds}
Let $U(1,n - 1)$, $U(1,n)$ be $(M,B,\varepsilon,\delta)$-rigid with coefficients $m_i, p_i$. Then for all $1 \leq i,j,k,l \leq M$, and $c,d$ such that $|c|,|d| \leq 2$, the following statements are true with $N = n - 1$ if and only if they are true with $N = n$.

\begin{enumerate}
	\item $a_i(N) + a_j(N) + c > b_M(N)$.
	\item $b_i(N) + b_j(N) + c > b_M(N)$.
	\item $a_i(N) + a_j(N) + c > b_k(N) + b_l(N) + d$.
	\item $a_i(N) + a_j(N) + c < a_k(N) + a_l(N) + d$.
	\item $b_i(N) + b_j(N) + c > b_k(N) + b_l(N)+d$.
	\item $b_i(N) + b_j(N) + c > a_k(N) + a_l(N) + d$.
\end{enumerate}
\end{lemma}

\begin{proof}
The assertion $a_i(N) + a_j(N) + c > b_M(N)$ is equivalent to the assertion
	\begin{align*}
    (m_i + m_j - p_M)(N + B) > \delta_M - \varepsilon_i - \varepsilon_j - c.
    \end{align*}
    
\noindent If $m_i + m_j - p_M = 0$, then this is true for $N = n$ if and only if it is true for $N = n - 1$. If $m_i + m_j - p_M \geq 1$, then we note that
	\begin{align*}
    \left(m_i + m_j - p_M\right)(N + B) &\geq N + B \\
    &\geq 4\max(\varepsilon,\delta)-B+4 \\
    &> \delta_M - \varepsilon_i - \varepsilon_j - c, 
    \end{align*}
    
 \noindent so the assertion is true for $N = n$ and $N = n - 1$. If $m_i + m_j - p_M < 0$, then we have
 	\begin{align*}
    \left|\delta_M - \varepsilon_i - \varepsilon_j - c\right| &< \delta + 2\varepsilon + 2 \\
    &< N + B \\
    &< \left|(m_i + m_j - p_M)(N + B)\right|,
    \end{align*}
    
\noindent hence the assertion is false for $N = n$ and $N = n - 1$.
    
The other five statements follow similarly.
\end{proof}

Lemmas \ref{RecursiveIntervals} and \ref{EndpointBounds} together allow us to prove a powerful statement showing how rigidity of $U(1,n - 1)$ can be used to obtain rigidity of $U(1,n)$.

\begin{lemma}\label{OneElementRigidity}
Suppose $U(1,n - 1)$ is $(M,B,\varepsilon,\delta)$-rigid with coefficients $m_i, p_i$, and $U(1,n)$ is $(M - 1,B,\varepsilon,\delta)$-rigid with coefficients $m_i, p_i$. Then $U(1,n)$ is $(M,B,\varepsilon,\delta)$-rigid with coefficients $m_i, p_i$.
\end{lemma}

\begin{proof}
By Lemma \ref{RecursiveIntervals}, we know that we can find indices $1 \leq i\leq j < M$ and integer $c$ such that $a_M(n - 1) = a_i(n - 1) + a_j(n - 1) + c$ or $a_M(n - 1) = b_i(n - 1) + b_j(n - 1) + c$, and $|c| \leq 2$. Define
	\begin{align*}
    x(t) &= \begin{cases} a_i(t) + a_j(t) + c & \text{if } a_M(n - 1) = a_i(n - 1) + a_j(n - 1) + c \\ b_it(t) + a_j(t) + c & \text{if } a_M(n - 1) = b_i(n - 1) + b_j(n - 1) + c \end{cases}.
    \end{align*}
    
\noindent We seek to prove that $x(n) = a_M(n)$. First, note that $x(n - 1) > b_{M - 1}(n - 1)$, so by Lemma \ref{EndpointBounds} we know that $x(n) > b_{M - 1}(n)$.

Suppose that $a_M(n) < x(n)$. By Lemma \ref{RecursiveIntervals}, we could then find indices $1 \leq i\leq j < M$ such that
	\begin{align*}
    a_M(n) = \begin{cases} a_k(n) + a_l(n) + c \\ b_k(n) + b_l(n) + d \end{cases},
    \end{align*}
    
\noindent and so by Lemma \ref{EndpointBounds} we would conclude $a_M(n - 1) < x(n - 1)$, which is a contradiction. Therefore, $a_M(n) \geq x(n)$, and it shall suffice to prove that $x(n) \in U(1,n)$.

By its construction, it is evident $x(n)$ has at least one representation. We need to show that it doesn't have any others. This will happen only if there are indices $1 \leq k \leq l < M$ such that $(i,j) \neq (k,l)$ and $x(n) \in A_k(n) + A_l(n)$, or equivalently
	\begin{align*}
    a_k(n) + a_l(n) - 1 < x(n) < b_k(n) + b_l(n) + 1.
    \end{align*}
    
\noindent However, this is impossible, as it would imply by Lemma \ref{EndpointBounds} that
	\begin{align*}
    a_k(n - 1) + a_l(n - 1) - 1 < x(n - 1) < b_k(n - 1) + b_l(n - 1) + 1,
    \end{align*}
    
\noindent contradicting the fact that $x(n - 1) \in U(1,n - 1)$. Ergo, $x(n) \in U(1,n)$.

It remains to show that $b_M(n) = (N + B)p_M + \delta_M$. We know $b_M(n) + 1 \notin U(1,n)$, which means that there must be $1 \leq i \leq j < M$ such that $b_M(n) + 1 \in A_i(n) + A_j(n)$. However, since $b_M(n) \notin A_i(n) + A_j(n)$, it must be that
	\begin{align*}
    b_M(n) + 1 = \begin{cases} a_i(n) + a_j(n) & \text{if } i \neq j \\ \begin{cases} 2a_i(n) + 1 \\ 2a_i(n) + 2 \end{cases} & \text{if } i = j \end{cases}.
    \end{align*}
    
\noindent With this in mind, define a linear function
	\begin{align*}
    y(t) = \begin{cases} a_i(t) + a_j(t) - 1 & \text{if } b_M(n) = a_i(n) + a_j(n) - 1 \\ 2a_i(t) & \text{if } b_M(n) = 2a_i(n) \\  2a_i(t) + 1 & \text{if } b_M(n) = 2a_i(n) + 1 \end{cases}.
    \end{align*}
    
\noindent We wish to prove that $y(n - 1) = b_M(n - 1)$. First, we note that certainly $y(n - 1) \in U(1,n - 1)$---by construction, it has at least one representation, and it cannot have a second representation without $y(n)$ have a corresponding representation.

Furthermore, $y(n - 1) \geq b_M(n - 1)$---since $y(n) + 1 \in A_i(n) + A_j(n)$, it must be that $y(n - 1) + 1 \in A_i(n - 1) + A_j(n - 1)$, which shows $y(n - 1) \notin U(1, n - 1)$. Notice that if $y(n - 1) > b_M(n - 1)$, then in fact $y(n - 1) > b_M(n - 1) + 1$.

However, $b_M(n - 1) + 1 \notin U(1,n - 1)$, hence there must exist $1 \leq j \leq k < M$ such that $b_M(n - 1) + 1 \in A_j(n - 1) + A_k(n - 1)$, and in fact, by the same argument as above, we must have
	\begin{align*}
    b_M(n - 1) + 1 = \begin{cases} a_k(n - 1) + a_l(n - 1) & \text{if } i \neq j \\ \begin{cases} 2a_k(n - 1) + 1 \\ 2a_l(n - 1) + 2 \end{cases} & \text{if } i = j \end{cases}.
    \end{align*}
    
\noindent Ergo, by Lemma \ref{EndpointBounds}, since $b_M(n) + 1 \nless y(n)$, it follows that $b_M(n - 1) + 1 \nless y(n - 1)$. We conclude that $y(n - 1) = b_M(n - 1)$, as desired.
\end{proof}

The proof of Theorem \ref{Rigidity for a=1} now falls out immediately.

\begin{proof}[Proof of Theorem \ref{Rigidity for a=1}]
We are given that $U(1,N_0)$ is $(M,B,\varepsilon,\delta)$-rigid. Inducting on $N$, we can assume that $U(1,n - 1)$ is $(M,B,\varepsilon,\delta)$-rigid, and it suffices to prove that $U(1,n)$ is $(M,B,\varepsilon,\delta)$-rigid.

We prove inductively that $U(1,n)$ is $(M',B,\varepsilon,\delta)$ for all $1 \leq M' \leq M$. It is clear that $U(1,n)$ is $(2,B,\varepsilon,\delta)$-rigid. By Lemma \ref{OneElementRigidity}, we know that if $U(1,n)$ is $(M' - 1,B,\varepsilon,\delta)$-rigid, then it is $(M',B,\varepsilon,\delta)$-rigid. This concludes the proof.
\end{proof}

\section{Applications of the Rigidity Conjecture:}\label{Applications}
We now give two applications of the conjectures \ref{Rigidity Conjecture} and \ref{NumericalRigidity}. In both cases, while we can prove some partial results unconditionally, we get further information if we can assume something about the rigidity of Ulam sequences.

\subsection{Regular Ulam Sequences} It was proved by Finch \cite{finch_1991, finch_1992_1, finch_1992_2} that if a $(2,1)$-Ulam set contains finitely many even terms, then it is \emph{regular}---that is, the differences between consecutive terms are eventually periodic. It is conjectured that a $(2,1)$-Ulam set $U(a,b)$ with $a<b$ coprime contains finitely many even terms if and only if

	\begin{enumerate}
		\item $a = 2$, $b \geq 5$,
        \item $a = 4$,
        \item $a = 5$, $b = 6$, or
        \item $a \geq 6$ and $a$ or $b$ is even.
	\end{enumerate}
    
\noindent Schmerl and Spiegel \cite{schmerl_spiegel_1994} proved the $a = 2$, $b \geq 5$ case; Cassaigne and Finch \cite{cassaigne_finch_1995} proved the case where $a = 4$, $b \equiv 1 \mod 4$. Our goal is to give a general proof technique for demonstrating that a given Ulam sequence has finitely many even terms, and therefore regular. Let $1_{U(a,b)}$ be the indicator function of $U(a,b)$. Given a positive integer $l$ and a positive odd number $k$, define
	\begin{align*}
    S^l_{a,b}(k) = \left\{1_{U(a,b)}(k + 2i)\right\}_{i = 0}^{l - 2}.
    \end{align*}

\begin{theorem}\label{Regular Theorem}
Let $l, a, b$ be positive integers, and $p < q$ be positive odd integers such that $q \geq 2l$, $a < b < 2l - 2$, $S^l_{a,b}(p) = S^l_{a,b}(q)$, and
	\begin{align*}
    U(a,b) \cap 2\ZZ \cap [2l, 3q - p] = \emptyset.
    \end{align*}
    
\noindent Then
	\begin{align*}
    U(a,b) \cap 2\ZZ \cap [2l,\infty) = \emptyset.
    \end{align*}
\end{theorem}

\begin{corollary}\label{Regular Sequences}
For integer pairs $(a,b)$ given below, $U(a,b)$ is regular.
	\begin{align*}
    \begin{array}{lllllll}
    (4, 11) & (4, 19) & (6, 7) & (6, 11) & (7, 8) & (7, 10) & (7, 12) \\
    (7, 16) & (7, 18) & (7, 20) & (8, 9) & (8, 11) & (9, 10) & (9, 14) \\
    (9, 16) & (9, 20) & (10, 11) & (10, 13) & (10, 17) & (11, 12) & (11, 14) \\
    (11, 16) & (11, 18) & (11, 20) & (12, 13) & (12, 17) & (13, 14)
    \end{array}
    \end{align*}
\end{corollary}

\begin{proof}
By direct computation, we find triples $(l, p, q)$ satisfying the conditions of Theorem \ref{Regular Theorem}.
	
\begin{minipage}[t]{.5\textwidth}
	\begin{align*}
    	\begin{array}{l|l}
        (a,b) & (l, p, q) \\ \hline
        (4, 11) & (25, 107, 1425) \\
		(4, 19) & (41, 14745, 17305) \\
		(6, 7) & (57,  8537,  70987) \\
		(6, 11) & (89,  1032425,  1033833) \\
		(7, 8) & (71,  14331,  57089) \\
		(7, 10) & (85,  95587,  102181) \\
		(7, 12) & (99,  79423,  80991) \\
		(7, 16) & (127,  46957,  47965) \\
		(7, 18) & (141,  196513,  198753) \\
		(7, 20) & (155,  50893,  52125) \\
		(8, 9) & (91,  1037093,  1038533) \\
		(8, 11) & (111,  2125501,  4308725) \\
		(9, 10) & (109,  117117,  747935) \\
        (9, 14) & (145,  558073,  560377)
        \end{array}
	\end{align*}
\end{minipage}%
\begin{minipage}[t]{0.5\textwidth}
	\begin{align*}
    	\begin{array}{l|lcl|l}
        (a,b) & (l, p, q) \\ \hline
        (9, 16) & (163,  60093,  65277) \\
		(9, 20) & (199,  219761,  222929) \\
		(10, 11) & (133,  470303,  485615) \\
		(10, 13) & (157,  5804601,  5807097) \\
		(10, 17) & (205,  3919981,  3933037) \\
		(11, 12) & (155,  140511,  142975) \\
		(11, 14) & (177,  507965,  509373) \\
		(11, 16) & (199,  394379,  400715) \\
		(11, 18) & (221,  29995,  37035) \\
		(11, 20) & (243,  46291,  54035) \\
		(12, 13) & (183,  3329465,  3330921) \\
		(12, 17) & (239,  3204117,  3211733) \\
		(13, 14) & (209,  1421023,  1427679)
        \end{array}
    \end{align*}
\end{minipage}

\end{proof}

The drawback of Theorem \ref{Regular Theorem} is that it requires separate computations for every Ulam sequence $U(a,b)$, and so for example it is insufficient to prove that $U(4,b)$ is regular for every $b \equiv -1 \mod 4$. Conjecture \ref{Rigidity Conjecture} strengthens Theorem \ref{Regular Theorem} considerably, as it implies that if we can find integers $p,q,l$ satisfying the conditions of the theorem for a sufficiently large $b \equiv -1 \mod 4$, then  every sequence $U(4,b')$ with $b' \equiv - 1 \mod 4$ and $b' \geq b$ has only finitely many even terms.

To prove Theorem \ref{Regular Theorem}, we start with a useful lemma that establishes that if it is false, then there is a bijective correspondence between odd Ulam numbers in different intervals.

\begin{lemma}\label{Regular Lemma}
Let $l, a, b$ be positive integers, and $p < q$ be positive odd integers such that $q \geq 2l$, $a < b < 2l - 2$, $S^l_{a,b}(p) = S^l_{a,b}(q)$,
	\begin{align*}
    U(a,b) \cap 2\ZZ \cap [2l, 3q - p] = \emptyset
    \end{align*}
    
\noindent and
	\begin{align*}
    U(a,b) \cap 2\ZZ \cap [2l,\infty) \neq \emptyset.
    \end{align*}
    
\noindent Let $\tilde{u}$ be the smallest even number in $U(a,b)$ greater than $3q - p$. Then there is a well-defined bijection
	\begin{align*}
    U(a,b) \cap (1 + 2\ZZ) \cap [p, \tilde{u} + p - q - 1] &\rightarrow U(a,b) \cap (1 + 2\ZZ) \cap [q, \tilde{u} - 1] \\
    u &\mapsto u + q - p.
    \end{align*}
\end{lemma}

\begin{proof}
We will show that there is a well-defined bijection
	\begin{align*}
    \phi_m: U(a,b) \cap (1 + 2\ZZ) \cap [p, p + 2m] &\rightarrow U(a,b) \cap (1 + 2\ZZ) \cap [q, q + 2m] \\
    u &\mapsto u + q - p,
    \end{align*}
    
\noindent for all integers $0 \leq m \leq \frac{\tilde{u} - q - 1}{2}$. We know that $S^l_{a,b}(p) = S^l_{a,b}(q)$, hence $p + 2m' \in U(a,b)$ if and only if $q + 2m' \in U(a,b)$ for all $0 \leq m' \leq l - 2$, which proves the claim for $m \leq l - 2$.

For all other $m$, we apply induction---that is, let $l - 2 < h \leq \frac{\tilde{u} - q - 1}{2}$ such that $\phi_{h - 1}$ is a bijection. We need to show that $\phi_h$ is bijection. This is equivalent to proving that $p + 2h \in U(a,b)$ if and only if $q + 2h \in U(a,b)$. Define sets
	\begin{align*}
    P &= \left\{(u,v) \in U(a,b)^2 \middle|u \equiv 0 \mod 2, \  v \equiv 1 \mod 2, \ u + v = p + 2h\right\} \\
    Q &= \left\{(u,v) \in U(a,b)^2 \middle|u \equiv 0 \mod 2, \ v \equiv 1 \mod 2, \ u + v = q + 2h\right\},
    \end{align*}
    
\noindent which enumerate the number of representations of $p + 2h$ and $q + 2h$, respectively. If we can show that $|P| = |Q|$, then this will imply that $p + 2h \in U(a,b)$ if and only if $q + 2h \in U(a,b)$. However, we can construct a bijection between these two sets by
\begin{align*}
\psi: P &\rightarrow Q\\
(u,v) &\mapsto \left(u, \varphi_{h - 1}(v)\right) \\
	&= (u, v + q - p).
\end{align*}

\noindent This is well-defined since $u + v = p + 2h$ implies $v \leq p + 2h - 1$.
\end{proof}

\begin{proof}[Proof of Theorem \ref{Regular Theorem}]
We argue by contradiction. That is, suppose that there exist even Ulam numbers larger than $3q - p$. Let $\tilde{u}$ be the smallest such element. We know $\tilde{u} = u_1 + u_2$ for some $u_1 < u_2 \in U(a,b)$. Every even Ulam number less than $\tilde{u}$ is smaller than $2l$, hence one of $u_1, u_2$ is odd---otherwise, we have
	\begin{align*}
    u_1 + u_2 &< 4l \leq 3q - p,
    \end{align*}
    
\noindent which is a contradiction. Since $\tilde{u}$ is even, we conclude that $u_1, u_2$ are both odd. Next, we show that $\tilde{u} - q + p$ has at least two representations as the sum of two distinct elements of $U(a,b)$. Note that
	\begin{align*}
	\tilde{u} - q + p &\geq (3q - p) - q + p = 2q > 2l,
    \end{align*}
    
\noindent and since $\tilde{u} - q + p$ is even, this implies it is not in $U(a,b)$. Consequently, it will suffice to prove that it has at least one representation. Note that
	\begin{align*}
    u_2 &> \frac{\tilde{u}}{2} > \frac{3q - p}{2} > q \\
    u_2 &\leq \tilde{u} - 1,
    \end{align*}

\noindent so by Lemma \ref{Regular Lemma}, since $u_2 \in U(a,b)$ it follows $u_2 + q - p \in U(a,b)$. Therefore, $\tilde{u} + q - p = u_1 + (u_2 + q - p)$ is a representation.
    
Write
	\begin{align*}
    \tilde{u} - q + p = v_1 + v_2 = v_1' + v_2',
    \end{align*}
    
\noindent where $v_1 < v_2, v_1' < v_2' \in U(a,b)$. Note that $v_2 > q$, since
	\begin{align*}
    v_2 &> \frac{\tilde{u} - q + p}{2} > \frac{(3q - p) - q + p}{2} > q.
    \end{align*}
    
\noindent Similarly, $v_2' > q$. From this it follows that $v_2, v_2' > 2l$, and we conclude that $v_2, v_2'$ must be odd. Finally, note that
	\begin{align*}
    p < q < v_2, v_2' \leq \tilde{u} + p - q - 1,
    \end{align*}
    
\noindent and therefore by Lemma \ref{Regular Lemma}, $v_2 + q - p, v_2' + q - p \in U(a,b)$, which is a contradiction since
	\begin{align*}
    \tilde{u} &= v_1 + (v_2 + q - p) \\ &= v_1' + (v_2' + q - p).
    \end{align*}
    
\noindent This concludes the proof.
\end{proof}

\subsection{Density in the $U(1,n)$ Sequence} Knowing the asymptotic structure of $U(1,n)$ gives insight into the density of $U(1,n)$ for large $n$. To see this, let us assume that Conjecture \ref{NumericalRigidity} and define
	\begin{align*}
    \delta_M(n) &= \frac{\left|\bigsqcup_{i = 1}^M [a_i(n), b_i(n)] \cap \ZZ\right|}{b_M(n)} \\
    &= \sum_{i = 1}^M \frac{(n + B)(p_i - m_i) + (\delta_i - \varepsilon_i)}{(n + B)p_M + \delta_M},
    \end{align*}
    
\noindent which is the density of $U(1,n)$, truncated to the first $M$ intervals. Letting $\delta(n)$ be the density of $U(1,n)$, we see that
	\begin{align*}
    \lim_{n \rightarrow \infty} \delta(n) &= \lim_{n \rightarrow \infty}\lim_{M \rightarrow \infty} \delta_M(n) \\
    &= \lim_{M \rightarrow \infty}\lim_{n \rightarrow \infty} \delta_M(n) \\
    &= \lim_{M \rightarrow \infty}\sum_{i = 1}^M \frac{p_i - m_i}{p_M}.
    \end{align*}

\noindent Note that $p_i - m_i$ is either $0$ or $1$---this is, for example, a consequence of Lemma \ref{sieve2}, which appears later in this section. Therefore, we get an upper bound on the asymptotic density.
	\begin{align*}
    \lim\limits_{n\to\infty}\delta(n)\leq\lim\limits_{M\to\infty}\frac{M}{p_M}.
    \end{align*}

\noindent Without appealing Conjecture \ref{NumericalRigidity}, it is still possible to give an explicit upper bound on the density of $U(1,n)$ by studying its structure. Specifically, we seek to prove the following theorem.

\begin{theorem}\label{upper bound}
Let $n \geq 2$ and let $I$ be a set of $3n$ consecutive positive integers greater than $2n + 2$. Then $|I\cap U(1,n)|\leq n + 1$.
\end{theorem}

As an immediate corollary of this theorem, we obtain an upper bound on the density.

\begin{corollary}\label{ExplicitUpperBound}
$\delta(n) \leq \frac{n+1}{3n}$.
\end{corollary}

\begin{proof}
Partition the first $k$ integers greater than $2n + 2$ into runs of $3n$ consecutive integers. Each such partition contains at most $n + 1$ terms of $U(1,n)$. The proportion of Ulam numbers less than or equal to $k$ is then no bigger than
	\begin{align*}
    \frac{(n + 1)(\frac{k}{3n} + 1) + 2n + 2}{k} = \frac{n + 1}{3n}\left(1 + \frac{1}{k}\right) + \frac{2n + 2}{k}.
    \end{align*}
    
\noindent In the limit, we get the desired upper bound.
\end{proof}

It should be noted that this is likely not a tight upper bound---asymptotically,
\begin{align*}
\frac{n+1}{3n}\approx\frac{1}{3},
\end{align*}

\noindent but numerical data for $n \geq 4$ suggests that the actual density is $\approx 1/6$. We will give an improvement on this upper bound for the special case $U(1,2)$ at the end of this section.

Before we prove Theorem \ref{upper bound}, we give a couple useful lemmas. First, we note that the statement of Lemma \ref{Initial Segment}, which only implies a statement for Ulam sequences $U(1,n)$ with $n$ sufficiently large, can be made completely explicit in this case.

\begin{lemma}\label{Initial Segment Explicit}
Let $n \geq 2$. The first three intervals of $U(1,n)$ are $\{1\}$, $[n,2n]\cap \ZZ$, and $\{2n+2\}$.
\end{lemma}

\begin{proof}
Clearly, all elements of the form $n+i$ for $1 \leq i \leq n$ have the unique Ulam representation $n+i = (n+i-1) + 1$. However, $2n+1 \notin U(1,n)$, because it has a second Ulam representation $n + (n+1)$. Finally, $2n+2 = n + (n+2)$, which is its only Ulam representation, and $2n + 3 \notin U(1,n)$ since $2n + 3 = (2n + 2) + 1 = n + (n + 3)$.
\end{proof}

\begin{lemma}\label{sieve1}
If $a,a + k \in U(1,n)$ for some $1\leq k\leq n$, then $[a + k + n, a + 2n] \cap \ZZ \subset \ZZ \ \backslash \ U(1,n)$.
\end{lemma}

\begin{proof}
Every integer in this interval is of the form $a+k+n+i$ for $0\leq i\leq n - k$, hence it has at least two Ulam representations: $(a+k)+(n+i)$ and $a + (n+k+i)$, where we have used the fact that $n + i, n + k + i \in [n, 2n]$, and hence are in the Ulam sequence by Lemma \ref{Initial Segment Explicit}.
\end{proof}

\begin{lemma}\label{sieve2}
Let $1 \leq k \leq n$. If $[a,a + k] \cap \ZZ \subset U(1,n)$, then $[a + n + 1, a + k + 2n - 1] \cap \ZZ \subset \ZZ \ \backslash \ U(1,n)$.
\end{lemma}
\begin{proof}
We partition
\begin{align*}
[a,a+k] \cap \ZZ = \bigcup\limits_{i=0}^{k-1} [a+i, a+i+1] \cap \ZZ,
\end{align*}

\noindent and so it suffices to prove the claim with $k = 1$, which is an immediate corollary of Lemma \ref{sieve1}.
\end{proof}

Lemma \ref{sieve2} shows that if there are long runs of consecutive elements in the Ulam sequence, then there must be longer run of consecutive elements later on that do not belong to the Ulam sequence. With this observation in hand, we proceed to the proof of Theorem \ref{upper bound}.

\begin{proof}[Proof of Theorem \ref{upper bound}]
If $I \cap U(1,n) = \emptyset$, we are done. Otherwise, let $a > 2n + 2$ be the smallest element in $I \cap U(1,n)$. There are two cases: either $[a, a + n - 1]$ contains at least two consecutive elements $u, u + 1 \in U(1,n)$, or it does not. We consider these cases separately.

\begin{case} \ 

Since we are given that $[a, a+n-1]\cap U(1,n)$ contains at least two consecutive elements, we can partition it into disjoint intervals 
	\begin{align*}
    [a, a+ n - 1] \cap U(1,n) &= \bigsqcup_{i = 1}^m [a + k_i, a + l_i] \cap \ZZ \\
    &= \bigsqcup_{j = 1}^t \{a + c_j\}
    \end{align*}
    
\noindent such that $k_i \leq l_i + 1 < k_{i+1}$, $c_j + 1 < c_{j + 1}$, and for no $i,j$ is $c_j \in [k_i - 1, l_i + 1]$. By Lemma \ref{sieve2}, $[a+n+k_i+1, a+l_i+2n-1] \subset \ZZ \ \backslash \ U(1,n)$ for $1\leq i\leq m$. Note that since $k_m \leq n - 1$ and $l_1 \geq 1$, we have $a+n+k_m+1 \leq a+l_1+2n-1$, and hence
	\begin{align*}\bigcup_{i=1}^m [a+n+k_i+1, a+l_i+2n-1] \cap \ZZ &= [a+n+k_1+1, a+2n+l_m-1] \cap \ZZ \\ &\subset \ZZ \ \backslash \ U(1,n).
    \end{align*}
    
\noindent Therefore,
	\begin{align*}
    I \cap U(1,n) \subset &\left([a, a + n + k_1] \cap \ZZ\right) \\
    \cup &\left([a + 2n + l_m, a + 3n - 1] \cap \ZZ\right).
    \end{align*}

\noindent However, we claim that
	\begin{align*}
    \left|[a + 2n + l_m, a + 3n - 1] \cap U(1,n)\right| + \left|[a + l_m, a + n - 1] \cap U(1,n)\right| \leq n - l_m.
    \end{align*}
    
\noindent It suffices to prove this assuming that $[a + l_m, a + n - 1] \cap U(1,n) \neq \emptyset$---let $u_1, u_2, \ldots u_s$ be the Ulam numbers in $[a + l_m, a + n - 1]$. If $s = 1$, then we note that
	\begin{align*}
    a + 2n + l_m = (a + l_m) + 2n = u_1 + (2n - (u_1 - a - l_m)),
    \end{align*}
    
\noindent and as this gives two representations, it must be that $a + 2n + l_m \notin U(1,n)$. If $s > 1$, note that for every $1 \leq i < j \leq s$, by Lemma \ref{sieve1},
	\begin{align*}
    [u_j + n, u_i + 2n] \cap \ZZ \subset \ZZ \ \backslash \ U(1,n),
	\end{align*}
    
\noindent hence
	\begin{align*}
    [a + 2n + l_m, u_{s - 1} + 2n] \cap \ZZ \subset \ZZ \ \backslash \ U(1,n).
    \end{align*}
    
 \noindent Note that
 	\begin{align*}
    \left|[a + 2n + l_m, u_{s - 1} + 2n] \cap \ZZ\right| \geq s
	\end{align*}
 
\noindent unless $u_{s - 1} = a + l_m + s - 1$, which is to say that $[a + l_m, a + l_m + s - 1] \subset U(1,n)$. But by the definition of $l_m$, it can only be that $a + l_m \in U(1,n)$ if $l_m = n - 1$, which is not possible since we assumed that there are at least two Ulam numbers in $[a + l_m, a + n - 1]$. As desired, we conclude that
	\begin{align*}
    \left|[a + l_m, a + n - 1] \cap U(1,n)\right| + \left|[a + 2n + l_m, a+ 3n - 1] \cap U(1,n)\right|  \leq n - l_m,
    \end{align*}
    
\noindent and therefore
	\begin{align*}
    \left|I \cap U(1,n)\right| &\leq \left|[a + l_m, a + n - 1] \cap U(1,n)\right| \\
    &+ \left|[a + n, a + n + k_1] \cap U(1,n)\right| \\
    &+ \left|[a + 2n + l_m, a+ 3n - 1] \cap U(1,n)\right| \\
    &\leq n - l_m + k_1 - 1 \\
    &\leq n - 1.
    \end{align*}
\end{case}

\begin{case} \

In this case, we are given that 
	\begin{align*}
    [a, a+ n - 1] \cap U(1,n) &= \bigsqcup_{j = 1}^t \{a + c_j\}
    \end{align*}
    
\noindent where $c_j + 1 < c_{j + 1}$. This implies that for $k > j$,
	\begin{align*}
    k - j < c_k - c_j < n.
    \end{align*}
    
\noindent By Lemma \ref{sieve1}, we have
	\begin{align*}
    [a + c_k + n, a + c_j + 2n] \cap \ZZ \subset \ZZ \ \backslash \ U(1,n),
    \end{align*}
    
\noindent and consequently,
	\begin{align*}
    [a + c_2 + n, a + c_{t - 1} + 2n] \cap \ZZ &= \bigcup_{1 \leq i < j \leq t}[a + c_k + n, a + c_j + 2n] \cap \ZZ \\
    &\subset \ZZ \ \backslash \ U(1,n).
    \end{align*}
    
\noindent Ergo,
	\begin{align*}
    \left|I \cap U(1,n)\right| &= \left|[a,a + n - 1] \cap U(1,n)\right| \\
    &+ \left|[a + n, a + c_2 + n - 1]\cap U(1,n)\right| \\
    &+ \left|[a + c_2 + n, a + c_{t - 1} + 2n]\cap U(1,n)\right| \\
    &+ \left|[a + c_{t - 1} + 2n + 1, a + 3n - 1]\cap U(1,n)\right| \\
    &\leq t + c_2 + n - c_{t - 1} - 1 \\
    &\leq n + 1.
    \end{align*}
\end{case}

\noindent This concludes the proof.
\end{proof}

For $n = 2$, Corollary \ref{ExplicitUpperBound} gives an upper bound of $\frac{1}{2}$ on the density. Using somewhat different techniques to the proof of Theorem \ref{upper bound}, we can improve this upper bound to $6/17\approx 0.353$.

\begin{theorem}
The density of $U(1,2)$ is at most $6/17$.
\end{theorem}

\begin{proof}
Let $a \in U(1,n)$ and define $I = [a, a + 8] \cap \ZZ$, $J = [a, a + 16] \cap \ZZ$. We claim that either $|I \cap U(1,2)| \leq 3$, or $|J \cap U(1,2)| \leq 6$. We make use of the fact that
	\begin{align*}
    1,2,3,4,6,8,11,13,16 \in U(1,2).
    \end{align*}
    
\noindent If $|I \cap U(1,2)| > 3$, then $I = \{a, a + 2, a + 5, a + 7\}$. Otherwise, $I \cap U(1,2)$ contains a pair of elements $u, u + 1$ such that $u + 1 = a + 2, a + 3, a + 4, a + 6$, or $a + 8$, which gives two representations; this is a contradiction.

In this case, $J \cap U(1,2) \subset \{a, a + 2, a + 5, a + 7, a + 12, a + 14\}$---otherwise, it contains an element with two representations. Consequently, $|J \cap U(1,2)| \leq 6$. This means we can now define two sequences $u_1, u_2, u_3, \ldots$, $L_1, L_2, L_3, \ldots$ recursively---let $u_1 = 1$ and $L_1 = 17$, and then define $u_{i + 1}$ to be the smallest element of the Ulam sequence larger than $u_i + L_i$, and
	\begin{align*}
    L_{i + 1} &= \begin{cases} 17 & \text{if } \left|[u_{i + 1}, u_{i + 1} + 16] \cap U(1,2)\right| \leq 6 \\ 9 & \text{otherwise} \end{cases}.
    \end{align*}
    
\noindent We can then partition the positive integers into sets of the form $[u_{i + 1}, u_{i + 1} + L_i]$ and $[u_{i + 1} + L_i + 1, u_{i + 2} - 1]$. The density of $U(1,2)$ in any of these sets is no more than $6/17$, and that implies that the density of $U(1,2)$ is bounded by $6/17$.
\end{proof}

\section{Classification of $(3,2)$-Ulam Sets:}\label{HigherDimensions}

Up until this point, we have only considered $(2,1)$-Ulam sets; we now turn to the problem of classifying higher dimensional Ulam sets. The classification problem for non-degenerate $(2,2)$-Ulam sets was solved by Kravitz and Steinerberger \cite{kravitz_steinerberger_2017}. In particular, they showed that after a linear transformation, the Ulam set becomes $U\left((1,0),(0,1)\right)$, illustrated in Figure $\ref{base lattice}$. We shall denote this set by $\A$.

\begin{figure}
	\begin{tabular}{ccc}
    \includegraphics[height = 0.2\textheight]{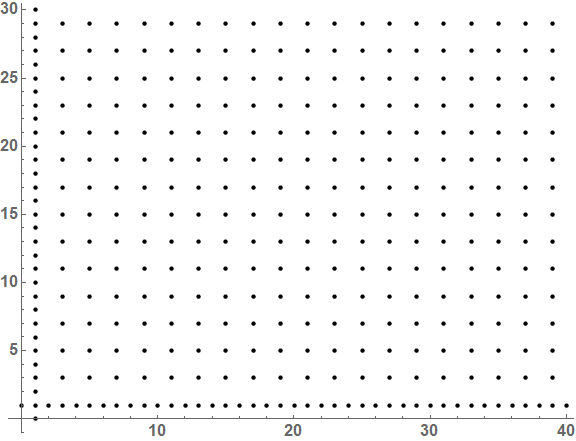} & & \includegraphics[height = 0.2\textheight]{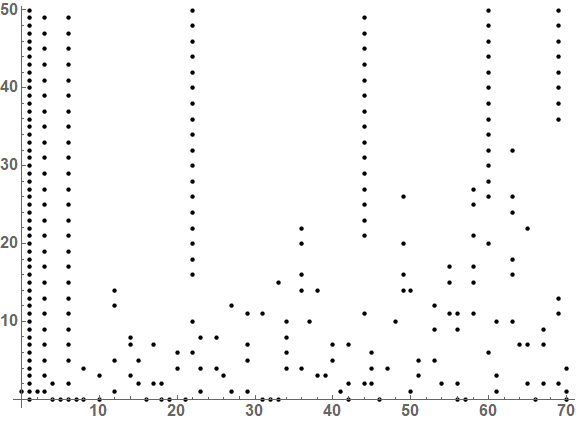}
    \end{tabular}
	\caption{The $(2,2)$-Ulam set $\A$, and the $(3,2)$-Ulam set $U_\A(4,0)$.}
    \label{base lattice}
\end{figure}

We shall consider $(3,2)$-Ulam sets that are extensions of such Ulam sets---that is, we shall assume that two of the basis vectors are $(1,0)$ and $(0,1)$. For convenience, we define
	\begin{align*}
    U_\A(v_1, v_2) &= U\left((1,0),(0,1),(v_1,v_2)\right) \\
    W_{(v_1, v_2)} &= \left\{(m,n) \in \ZZ_{\geq 0}^2 \middle| m < v_1 \text{ or } n < v_2 \right\} \\
    L_{(v_1, v_2)} &= \left\{(m,n) \in \ZZ_{\geq v_1} \times \ZZ_{\geq v_2}\right\}.
    \end{align*}
    
\noindent Note that if $(a,b) \in L_{(v_1, v_2)}$, then any representations it has have to lie in the set $W_{(v_1, v_2)}$. We use this fact to our advantage to prove the following lemma.

\begin{lemma}\label{Basic Structural Lemma}
Let $\U = U_\A(v_1, v_2)$ be a non-degenerate $(3,2)$-Ulam set with $v_1, v_2 \neq 0$. Then the following statements hold.
	\begin{enumerate}
    \item $v_1, v_2 > 1$ and at least one of $v_1, v_2$ is even.
    \item $\A \cap W_{(v_1, v_2)} = \U \cap W_{(v_1, v_2)}$.
    \item Every point $(m,n) \in \ZZ_{\geq 0}^2$ has at least one representation.
    \end{enumerate}
\end{lemma}

\begin{proof}
It was shown in \cite{kravitz_steinerberger_2017} that
	\begin{align*}
    \A = &\left\{(m,1) \middle| m \in \ZZ_{\geq 0}\right\} \cup \left\{(1,m) \middle| m \in \ZZ_{\geq 0}\right\} \\
    \cup &\left\{(2m + 1, 2n + 1) \middle| m,n \in \ZZ_{\geq 0}\right\}.
    \end{align*}
    
\noindent For $\U$ to be non-degenerate, it must be that $(v_1, v_2) \notin \A$, and since $v_1, v_2 \neq 0$, this implies that $v_1, v_2 > 1$ and at least one of $v_1, v_2$ is even.

All representations of points in $W_{(v_1, v_2)}$ are representations by elements in $\U$. It follows $\A \cap W_{(v_1, v_2)} = \U \cap W_{(v_1, v_2)}$.  However, this implies that
	\begin{align*}
    (m, n) = (m - 1, 1) + (1, n - 1)
    \end{align*}
    
 \noindent is a representation of $(m,n)$.
\end{proof}

We shall call $(m, n) = (m - 1, 1) + (1,n - 1)$ the \emph{standard representation} of $(m,n)$. By Lemma \ref{Basic Structural Lemma}, proving that $(m,n) \notin U_\A(v_1, v_2)$ for $v_1, v_2 \neq 0$ is equivalent to proving that it has a nonstandard representation. This makes working with Ulam sets of this form much simpler. On the other hand, if one of $v_1, v_2 = 0$, then the set $U_\A(v_1, v_2)$ has a copy of a $(2,1)$-Ulam set on either the $x$- or $y$-axis. An example of such a set is given in Figure \ref{base lattice}. Some partial results about such sets are given in \cite{kravitz_steinerberger_2017}, but in general describing their structure is an open problem.

We now give five examples of possible structures of sets $U_\A(v_1, v_2)$ with $v_1, v_2 \neq 0$, which are derived from numerical observations. An illustration of each of these five types is provided in Figure \ref{UlamTypes}.

\begin{definition}
Let $U \subset \ZZ_{\geq 0}^2$ and let $(v_1, v_2)$ be a vector in $U$. We say $U$ is of $L$ \emph{type for} $(v_1, v_2)$ if
	\begin{align*}
    U = &\{(v_1,v_2)\} \cup \left\{(m,1) \middle| m \in \ZZ_{\geq 0}\right\} \cup \left\{(1,m) \middle| m \in \ZZ_{\geq 0}\right\} \\
    &\cup \left\{(a + 2m v_1, b + 2m v_2) \middle| a,b,m \geq 0, \ a,b \in 1 + 2\ZZ, \ m \in \ZZ, \ (a,b) \in W_{(v_1,v_2)}\right\}.
    \end{align*}
    
\noindent We say $U$ is of \emph{column-deleted type for} $(v_1, v_2)$ if
	\begin{align*}
    U = & \{(v_1, v_2)\} \cup \left\{(m,1) \middle| m \in \ZZ_{\geq 0}\right\} \cup \left\{(1,m) \middle| m \in \ZZ_{\geq 0}\right\} \\
    &\cup \left\{(2m + 1, 2n + 1) \middle| m,n \in \ZZ_{\geq 0}, \ \text{if } 2m + 1 = v_1 + 1 \text{ then } 2n + 1 < v_2\right\}.
    \end{align*}
    
\noindent We say $U$ is of \emph{column-deleted} $L$ \emph{type for} $(v_1, v_2)$ if
	\begin{align*}
    U = &\{(v_1, v_2)\} \cup \left\{(m,1) \middle| m \in \ZZ_{\geq 0}\right\} \cup \left\{(1,m) \middle| m \in \ZZ_{\geq 0}\right\} \\
    &\cup \left\{(a + (m + 1) v_2 + 2, b + 2m + 5) \middle| a,b,m \geq 0, \ a,b,m \in 2\ZZ, \ a < m \text{ or } b = 0\right\}.
    \end{align*}
    
\noindent We say that $U$ is of \emph{shifted column-deleted type for} $(v_1, v_2)$ if
	\begin{align*}
    U = & \{(v_1, v_2)\} \cup \left\{(m,1) \middle| m \in \ZZ_{\geq 0}\right\} \cup \left\{(1,m) \middle| m \in \ZZ_{\geq 0}\right\} \\
    &\cup \left\{(m, n) \middle| m,n \geq 0, \ m < v_1, \ m,n \in 1 + 2\ZZ\right\} \\
    &\cup \left\{(m, n) \middle| m,n \geq 0, \ m > v_1, \ m \in 2\ZZ, \ n \in 1 + 2\ZZ\right\}.
    \end{align*}
    
\noindent We say $U$ is of \emph{exceptional type} if
	\begin{align*}
    U = &\{(v_1, v_2)\} \cup \{(8,8)\} \cup \left\{(m,1) \middle| m \in \ZZ_{\geq 0}\right\} \cup \left\{(1,m) \middle| m \in \ZZ_{\geq 0}\right\} \\
    &\cup \left\{(4, 2m + 4) \middle| m \in \ZZ_{\geq 0}\right\} \cup \left\{(2m + 4,4) \middle| m \in \ZZ_{\geq 0}\right\}.
    \end{align*}
\end{definition}

This list enumerates all the possibilities for sets $U_\A(v_1, v_2)$ if $v_1, v_2 \neq 0$.

\begin{theorem}\label{Type Theorem}
Let $\U = U_\A(v_1, v_2)$ be a non-degenerate $(3,2)$-Ulam set such that $v_1, v_2 \neq 0$. Then exactly one of the following is true of either $\U$ or its reflection about the line $y = x$.
\begin{enumerate}
	\item $v_1, v_2 \in 2\ZZ \cap [4,\infty)$ and $\U$ is of $L$ type.
    \item $v_1 \in 2\ZZ$, $v_2 \in \left(1 + 2\ZZ\right) \cap [4,\infty)$, and $\U$ is of column-deleted type.
    \item $v_1 \in 2\ZZ \cap [4,\infty)$, $v_2 = 2$, and $\U$ is of column-deleted $L$ type.
    \item $v_1 \in 2\ZZ$, $v_2 = 3$, and $\U$ is of shifted column-deleted type.
    \item $v_1 = v_2 = 2$ and $\U$ is of exceptional type.
\end{enumerate}
\end{theorem}

\begin{proof}
By Lemma \ref{Basic Structural Lemma}, the given list enumerates all possibilities for $v_1, v_2$, after accounting for a possible reflection around the $y = x$ line. Furthermore, it is easy to check that $\U \cap W_{(v_1, v_2)}$ is of the specified type in each case---that is, it is equal to the intersection of a set $U$ of the desired type with $W_{(v_1, v_2)}$.

Consider the case $v_1, v_2 \in 2\ZZ \cap [4,\infty)$. We shall show that $\U \cap W_{(a, b)}$ is of type $L$ for all $a,b \geq 0$. Note that by Lemma \ref{Basic Structural Lemma},
\begin{align*}
\A \cap W_{(3, 3)} &= \left\{(m,1) \middle| m \in \ZZ_{\geq 0}\right\} \\
&\cup \left\{(1,m) \middle| m \in \ZZ_{\geq 0}\right\} \\
&\cup \left\{(3,2m + 1) \middle| m \in \ZZ_{\geq 0} \right\} \\
&\cup \left\{(2m + 1,3) \middle| m \in \ZZ_{\geq 0} \right\} \\ &= \U \cap W_{(3, 3)}.
\end{align*}

\noindent It follows that if $(m,n) \in \U$ and $m, n > 1$, then $m, n \in 1 + 2\ZZ$. This is evident if $(m,n) \in W_{(3, 3)}$---otherwise, either $(m,n) = (k + 3, 2l + 2)$ or $(2l + 2, k + 3)$ for some $k, l \in \ZZ_{\geq 0}$, and we have nonstandard representations
	\begin{align*}
   (k + 3, 2l + 2) &= (3, 2l + 1) + (k, 1) \\
   (2l + 2, k + 3) &= (2l + 1, 3) + (1,k).
    \end{align*}
    
\noindent Furthermore, it must be that $\U \cap W_{(2v_1, 2v_2)}$ is of $L$ type. To see this, it suffices to show that
	\begin{align*}
    \U \cap W_{(2v_1, 2v_2)} \cap L_{(v_1, v_2)} = \{(v_1, v_2)\},
    \end{align*}
    
\noindent but as we know any point in this intersection must necessarily be of the form $(2m + 1, 2n + 1)$, we have a nonstandard representation 
	\begin{align*}
	(2m + 1, 2n + 1) &= (v_1, v_2) + (2m + 1 - v_1, 2n + 1 - v_2).
	\end{align*}
    
\noindent We now prove that $\U \cap W_{(2kv_1, 2kv_2)}$ is of $L$ type by inducting on $k \in \ZZ$---we have proved the base case $k = 1$, so it suffices to assume $\U \cap W_{(2mv_1, 2mv_2)}$ is $L$ type for some $m \in \ZZ_{\geq 0}$ and prove that $\U \cap W_{(2(m + 1)v_1, 2(m + 1)v_2)}$ is $L$ type. This amounts to proving that
	\begin{align*}
    \U \cap &W_{\left((2m + 1)v_1, (2m + 1)v_2\right)} \cap L_{(2mv_1, 2mv_2)} = \\ &W_{\left((2m + 1)v_1, (2m + 1)v_2\right)} \cap L_{(2mv_1, 2mv_2)} \cap \left(1 + 2\ZZ_{\geq 0}\right)^2 \\
    \U \cap &W_{\left((2m + 2)v_1, (2m + 2)v_2\right)} \cap L_{\left((2m + 1)v_1, (2m + 1)v_2\right)} = \emptyset.
    \end{align*}
    
\noindent This is easily proven by noting that the former set cannot possibly have any nonstandard representations, whereas the latter set is nothing more than
	\begin{align*}
    (v_1, v_2) + \U \cap W_{\left((2m + 1)v_1, (2m + 1)v_2\right)} \cap L_{(2mv_1, 2mv_2)}.
    \end{align*}

\noindent The other cases are similar.
\end{proof}

\section{Parity Restrictions on $(k,2)$-Ulam sets:}\label{Parity Section}

We close by giving a restriction on the possible structure of $(k,2)$-Ulam sets for $k \geq 2$. As in the previous section, we consider non-degenerate Ulam sets containing $(1,0), (0,1)$, and so we define
	\begin{align*}
    U_\A(v_1, v_2, \ldots v_n) = U\left((1,0),(0,1),v_1, \ldots v_n\right).
    \end{align*}
    
\noindent We shall show that the parity of any element in $U_\A(v_1, v_2, \ldots v_n)$ is eventually fixed, as long as none of the $v_i$ lie on the coordinate axes.

\begin{theorem}\label{Parity Theorem}
Let $\U = U_\A(v_1, v_2, \ldots v_n)$ be a non-degenerate $(n + 2, 2)$-Ulam set such that none of the $v_i$ lie on the coordinate axes. Then there exists a $v$ such that for all $u \in \U \cap L_v$, $u = v \mod 2$.
\end{theorem}

To prove Theorem \ref{Parity Theorem}, we first note that if $\U$ contains a point $(u_1, u_2)$ such that $(u_1, u_2 + 2k) \in \U$ for all $k \in \ZZ_{\geq 0}$, then for all $(u_1', u_2') \in \U \cap L_{(u_1, u_2)}$, $u_2 = u_2' \mod 2$. This is because if $u_2' \neq u_2 \mod 2$,
	\begin{align*}
    (u_1', u_2') = (u_1, u_2' - 1) + (u_1' - u_1, 1)
    \end{align*}

\noindent gives a nonstandard representation. It shall therefore suffice to prove the existence of such a point. Toward this end, we give a useful lemma.

\begin{lemma}\label{Infinite to Parity}
Let $\U = U_\A(v_1, v_2, \ldots v_n)$ be a non-degenerate $(n + 2, 2)$-Ulam set such that none of the $v_i$ lie on the coordinate axes. If there exists $m \in \ZZ_{> 1}$ such that there are infinitely many points of the form $(m,n) \in \U$, then there exists a point $(u_1, u_2)$ such that $(u_1, u_2 + 2k) \in \U$ for all $k \in \ZZ_{\geq 0}$.
\end{lemma}

\begin{proof}
Let $M \in \ZZ_{> 1}$ be the smallest $m$ such that there are infinitely many points of the form $(m,n) \in \U$. Note that in fact $M > 2$, since every element $(2,n)$ has at least two representations. Therefore, we can define $N$ be the largest $n$ such that $(m,n) \in \U$ where $1 < m < M$.

Consider any point $(M,n) \in \ZZ_{\geq 0}^2$ with $n > 2N$. For any representation of $(M,n)$, at least one of the summands must have $x$-coordinate $1$ or $M$---otherwise, the $y$-coordinates are too small to add up to $n$. If this representation is
	\begin{align*}
    (M,n) = (1,n') + (M - 1, n - n'),
    \end{align*}
    
\noindent then it is nonstandard if and only if $n - n' \neq 1$. However, if $n - n' \neq 1$, then every point $(M,n'')$ with $n'' > n$ has a nonstandard representation, which is impossible.

On the other hand, the only other possible representation is $(M,n) = (M,n - 1) + (0,1)$, so we conclude that $(M,n) \in \U$ if and only if $(M,n - 1) \notin \U$. We conclude that if we take
	\begin{align*}
    (u_1, u_2) = \begin{cases} (M,n) & \text{if } (M,n) \in \U \\ (M, n + 1) & \text{otherwise} \end{cases},
    \end{align*}
    
\noindent it satisfies the desired conditions.
\end{proof}

This is sufficient to prove Theorem \ref{Parity Theorem}.

\begin{proof}[Proof of Theorem \ref{Parity Theorem}]
We claim that there must exist some $m \in \ZZ_{> 1}$ such that there are infinitely many points of the form $(m,n) \in \U$. Suppose otherwise---then there must exist some strictly increasing function $\phi: \ZZ_{> 1} \rightarrow \ZZ_{> 1}$ such that if $(m,n) \in \U$ and $m,n > 1$, then $n < \phi(m)$.

Let $m > 2$ and $n > 2 \phi(m)$. Then if
	\begin{align*}
    (m,n) = (m_1, n_1) + (m_2, n_2)
    \end{align*}

\noindent is a representation of $(m,n)$, then it must be the standard representation---otherwise, $n_1 + n_2 < 2\phi(m) < n$. But this implies $(m,n) \in \U$, which is a contradiction.

Consequently, we can apply Lemma \ref{Infinite to Parity}. By our earlier remarks, we know there exists a point $(u_1, u_2) \in \U$ such that for all $(u_1',u_2') \in \U \cap W_{(u_1, u_2)}$, $u_2' \equiv u_2 \mod 2$.

On the other hand, the reflection of $\U$ about the $y = x$ is also an Ulam set, which we shall denote by $\mathcal{V}$. It is easy to check that $\mathcal{V}$ also satisfies the requirements of the theorem, and therefore must contain a point $(v_1, v_2)$ such that for all $(v_1', v_2') \in \mathcal{V} \cap W_{(v_1, v_2)}$, $v_2' \equiv v_2 \mod 2$. However, this means that if we take
	\begin{align*}
    v = \left(\max\{u_1, v_2\}, \max\{u_2, v_1\}\right),
    \end{align*}
    
\noindent then for all $u \in \U \cap L_v$, $u = v \mod 2$, as desired.
\end{proof}

\bibliography{UlamLibrary}
\bibliographystyle{alpha}
\end{document}